\documentclass{amsart}

\usepackage{amsmath}
\usepackage{amsthm}
\usepackage{amssymb}
\usepackage{mathrsfs}
\usepackage{enumerate}
\usepackage{bm}
\usepackage{tikz}
\usetikzlibrary{cd}
\usepackage{extpfeil}
\usepackage[normalem]{ulem}
\usepackage{caption}
\usepackage{subcaption}
\usepackage{bbold}
\usepackage{microtype}
\usepackage{hyperref}
\hypersetup{colorlinks,allcolors=black}
\usepackage[backend=biber, maxcitenames=7, maxbibnames=7, style=alphabetic, sorting=nty]{biblatex}
\newcommand{\ZZ}{\mathbb Z}
\newcommand{\FF}{\mathbb F}

\DeclareMathOperator{\Spec}{Spec}
\DeclareMathOperator{\Spc}{Spc}
\DeclareMathOperator{\Specinv}{Spec^\textup{inv}}
\newcommand{\ssm}{\smallsetminus}
\newcommand{\wtilde}{\widetilde}
\newcommand{\invlim}{\varprojlim}
\newcommand{\dirlim}{\varinjlim}
\renewcommand{\a}{\alpha}

\newcommand{\id}{\text{id}}
\DeclareMathOperator{\Ext}{Ext}
\newcommand{\cotensor}{\,\Box}
\DeclareMathOperator*{\colim}{colim}
\DeclareMathOperator{\cone}{cone}
\DeclareMathOperator{\Stable}{Stable}
\DeclareMathOperator{\Ch}{Ch}
\DeclareMathOperator{\thick}{thick}

\DeclareMathOperator{\im}{im}
\DeclareMathOperator{\End}{End}
\DeclareMathOperator{\supp}{supp}
\DeclareMathOperator{\res}{res}
\DeclareMathOperator{\incl}{incl}
\DeclareMathOperator{\ann}{ann}
\newcommand{\p}{\mathfrak p}
\newcommand{\q}{\mathfrak q}
\makeatletter
\renewcommand\dirlim{\mathop{\mathpalette\varlim@{\rightarrowfill@\scriptstyle}}\nmlimits@}
\renewcommand\invlim{\mathop{\mathpalette\varlim@{\leftarrowfill@\scriptstyle}}\nmlimits@}
\makeatother

\newtheorem{thm}{Theorem}[section]
\newtheorem*{thm*}{Theorem}
\newtheorem{prop}[thm]{Proposition}
\newtheorem*{prop*}{Proposition}
\newtheorem{lem}[thm]{Lemma}
\newtheorem*{lem*}{Lemma}
\newtheorem{cor}[thm]{Corollary}
\newtheorem*{cor*}{Corollary}

\theoremstyle{definition}
\newtheorem{defn}[thm]{Definition}

\theoremstyle{remark}
\newtheorem{rem}[thm]{Remark}

\begin{document}

\title[\resizebox{4.5in}{!}{Tensor triangular geometry of modules over the mod 2 Steenrod algebra}]{Tensor triangular geometry of modules over the mod 2 Steenrod algebra}
\author{Collin Litterell}

\begin{abstract}
    We compute the Balmer spectrum of a certain tensor triangulated category of comodules over the mod 2 dual Steenrod algebra. This computation effectively classifies the thick subcategories, resolving a conjecture of Palmieri.
\end{abstract}

\maketitle
\tableofcontents

\section{Introduction}

In \cite{dhs}, Devinatz, Hopkins, and Smith revolutionized the field of stable homotopy theory by proving the nilpotence theorem. One significant and seemingly immediate consequence proved in \cite{hs} was the thick subcategory theorem, which gave a deep structural understanding of the stable homotopy category.

The thick subcategory theorem in stable homotopy theory inspired many similar results in other fields. This was initiated by Hopkins in \cite{hopkins-global}, where he described an algebraic analogue of the thick subcategory theorem for the derived category of a commutative ring. Neeman gave a corrected proof for Noetherian rings in \cite{neeman-chromatic}, and Thomason proved the general case and extended the result to quasi-compact and quasi-separated schemes in \cite{thomason}. In modular representation theory, Benson, Carlson, and Rickard proved an analogue of the thick subcategory theorem for the stable module category of a finite $p$-group in \cite{bcr}, which was extended to finite group schemes by Friedlander and Pevtsova in \cite{fp}.

Let $A$ be the mod 2 dual Steenrod algebra. In \cite{pal-stable}, Palmieri extensively studied a particular triangulated category $\Stable(A)$ of comodules over $A$ using the framework of ``axiomatic stable homotopy theory.'' One of his many significant results was an analogue of nilpotence theorem, which says in part that restricting to a particular quotient Hopf algebra $D$ of $A$ detects nilpotence. However, an analogue of the thick subcategory theorem did not immediately follow. Instead, Palmieri gave an insightful conjecture regarding the thick subcategories. The main goal of this paper is to resolve Palmieri's conjecture by providing a classification of the thick subcategories of compact objects of $\Stable(A)$.

\begin{thm}
    There is a homeomorphism
        \[\Spc(\Stable(A)^c) \cong \Specinv(\Ext_D^{**}(\FF_2, \FF_2))\]
    between the Balmer spectrum of $\Stable(A)^c$ and the Zariski spectrum of bihomogeneous prime ideals in $\Ext_D^{**}(\FF_2, \FF_2)$ which are invariant under a certain coaction of the Hopf algebra $A \cotensor_D \FF_2$. Consequently, there is a bijection
        \[\left\{\begin{tabular}{@{}c@{}} thick subcategories \\ of $\Stable(A)^c$ \end{tabular}\right\} \stackrel{\sim}{\longleftrightarrow} \left\{\begin{tabular}{@{}c@{}} Thomason subsets of \\ $\Specinv(\Ext_D^{**}(\FF_2, \FF_2))$ \end{tabular}\right\}.\]
\end{thm}

This paper is organized as follows. In Section 2, we will provide more background on tensor triangular geometry, stable categories of comodules, and the dual Steenrod algebra. We will also briefly compare and contrast our situation with those of stable homotopy theory and modular representation theory. In Section 3, we will translate Hovey and Palmieri's classification result for finite-dimensional quotient Hopf algebras of $A$ into the language of tensor triangular geometry. In Section 4, we overcome the first major hurdle by extending Hovey and Palmieri's result to infinite-dimensional elementary quotient Hopf algebras of $A$ using the idea of continuity of the Balmer spectrum. In Section 5, we will compute $\Spc(\Stable(D)^c)$ from the Balmer spectra for the elementary quotient Hopf algebras of $A$ using the idea of Quillen stratification. Finally, in Section 6, we will compute $\Spc(\Stable(A)^c)$ using Palmieri's nilpotence and periodicity theorems. We also include an appendix on invariant prime ideals of Hopf algebroids, where we show that the invariant prime ideals form a spectral topological space.

\subsection{Acknowledgements}

I would like to thank John Palmieri for all of his guidance and support, for the many helpful conversations and comments, and for sharing his invaluable insights and expertise. Part of this work was carried out during the ``Spectral Methods in Algebra, Geometry, and Topology'' trimester program hosted by the Hausdorff Research Institute for Mathematics, which is funded by the Deutsche Forschungsgemeinschaft (DFG, German Research Foundation) under Germany's Excellence Strategy -- EXC-2047/1 -- 390685813.

\subsection{Conventions}

Any unadorned tensor products are either over the ground field $k$ (almost always $\FF_2$) or refer to the tensor product in a tensor triangulated category, which will be understood from context.

If an object is bigraded, the term \emph{homogeneous} is understood to mean bihomogeneous. All elements and ideals in (bi)graded rings or modules are assumed to be homogeneous.


\section{Preliminaries}

\subsection{Tensor triangular geometry}

A \emph{tt-category} (short for tensor triangulated category) is a triangulated category $\mathcal{C}$ with a symmetric monoidal structure $(\mathcal{C}, \otimes, \mathbb{1})$ which is compatible with the triangulated structure. In particular, the tensor product is exact in each variable. A functor between two tt-categories is called a \emph{tt-functor} if it is exact and strong symmetric monoidal.

Given a tt-category $\mathcal{C}$, a \emph{thick subcategory} of $\mathcal{C}$ is a full subcategory which is closed under isomorphisms, cofibers/cones, and retracts/direct summands. A \emph{tt-ideal} in $\mathcal{C}$ is a thick subcategory $\mathcal{I}$ of $\mathcal{C}$ such that $\mathcal{C} \otimes \mathcal{I} \subseteq \mathcal{I}$. A tt-ideal $\mathcal{I}$ is \emph{radical} if $a^{\otimes n} \in \mathcal{I}$ implies $a \in \mathcal{I}$, and it is \emph{prime} if it is proper and $a \otimes b \in \mathcal{I}$ implies $a \in \mathcal{I}$ or $b \in \mathcal{I}$. Given a set $S$ of objects in $\mathcal{C}$, let $\thick^\otimes\langle S \rangle$ denote the smallest tt-ideal containing $S$.

The motivation for classifying thick subcategories or tt-ideals comes from the often unattainable goal of classifying objects in $\mathcal{C}$ up to isomorphism. As a compromise, we can aim to classify objects in $\mathcal{C}$ up to isomorphism \emph{and} the operations coming from the tensor triangular structure. Note that $\thick^\otimes\langle X \rangle$ can be thought of as the collection of objects that can be built from $X$ using the basic operations in the tt-category $\mathcal{C}$ (e.g., taking cofibers or tensor products), so this question boils down to understanding when $\thick^\otimes\langle X \rangle = \thick^\otimes\langle Y \rangle$. Thus, classifying tt-ideals achieves our revised goal of classifying objects up to the tensor triangular structure.

Suppose $\mathcal{C}$ is an essentially small tt-category. The \emph{Balmer spectrum} of $\mathcal{C}$ is the set
    \[\Spc(\mathcal{C}) = \{\mathcal{P} \subsetneq \mathcal{C} : \mathcal{P} \text{ is a prime tt-ideal}\}.\]
The \emph{support} of an object $a \in \mathcal{C}$ is the subset
    \[\supp(a) = \{\mathcal{P \in \Spc(\mathcal{C}}) : a \notin \mathcal{P}\}.\]
The complements $U(a) = \Spc(\mathcal{C}) \ssm \supp(a)$ for $a \in \mathcal{C}$ form a basis for a topology on $\Spc(\mathcal{C})$. It turns out that computing $\Spc(\mathcal{C})$ as a topological space is equivalent to classifying the radical tt-ideals in $\mathcal{C}$ (see Theorems 16 and 17 in \cite{balmer-ttg} for precise statements). Moreover, if $\mathcal{C}$ is \emph{idempotent complete} (i.e., every idemopotent map splits), then every tt-ideal is radical and hence computing $\Spc(\mathcal{C})$ as a topological space is equivalent to classifying all tt-ideals in $\mathcal{C}$.

We also recall the \emph{comparison map} from \cite{balmer-sss}, which is a continuous map
    \[\rho_\mathcal{C} : \Spc(\mathcal{C}) \to \Spec^h(\End_\mathcal{C}^*(\mathbb{1}))\]
from the Balmer spectrum of $\mathcal{C}$ to the Zariski spectrum of homogeneous prime ideals in the graded endomorphism ring of the unit object in $\mathcal{C}$. This map is given by
    \[\mathcal{P} \mapsto \{f \in \End_\mathcal{C}^*(\mathbb{1}) : \cone(f) \notin \mathcal{P}\}.\]
Note that this extends easily to multigraded settings.

\subsection{Stable categories of comodules}

Given a graded commutative Hopf algebra $B$ over a field $k$, let $\Stable(B)$ denote the category whose objects are unbounded cochain complexes of injective graded left $B$-comodules and whose morphisms are cochain homotopy classes of bigraded maps. This is an axiomatic stable homotopy category in the sense of \cite{hps-axiomatic}, so in particular it is a tt-category. We note that $\Stable(B)$ is idempotent complete and rigidly-compactly generated. Moreover, there is a model category structure on the category $\Ch(B)$ of unbounded cochain complexes of left $B$-comodules such that the associated homotopy category is $\Stable(B)$. We refer the reader to \cite{hps-axiomatic} and \cite{pal-stable} for more details on $\Stable(B)$.

Recall the following notation:
    \begin{itemize}
        \item $S^0_B$ denotes the unit object in $\Stable(B)$, which is an injective resolution of the trivial comodule $k$.
        \item $[X, Y]^B_{**}$ denotes the morphisms between $X$ and $Y$ in $\Stable(B)$. Note that these are bigraded because we are working with cochain complexes of graded comodules; there is a homological grading and an internal grading. If $X$ and $Y$ are injective resolutions of $B$-comodules $M$ and $N$, respectively, then
            \[[X, Y]^B_{**} \cong \Ext_B^{**}(M, N).\]
        \item $\pi^B_{**}(X) = [S^0_B, X]^B_{**}$ denotes the homotopy groups of $X$. In particular,
            \[\pi^B_{**}(S^0_B) \cong \Ext_B^{**}(k, k).\]
        \item $E_*(X) = [S^0_B, E \otimes X]^B_{**}$ denotes the $E$-homology of $X$.
        \item $\Stable(B)^c$ denotes the full subcategory of compact objects, which is essentially small. Moreover, $\Stable(B)^c$ is equal to the thick subcategory generated by $S^0_B$, so as a consequence, every thick subcategory of $\Stable(B)^c$ is automatically a tt-ideal.
    \end{itemize}

For each conormal quotient Hopf algebra $C$ of $B$, there is a commutative associative ring object $HC \in \Stable(B)^c$ which is defined to be an injective resolution of the $B$-comodule $B \cotensor_C k$. Consequently, there is a Hopf algebroid
    \[\left(\pi^B_{**}(HC), \pi^B_{**}(HC \otimes HC)\right)\]
with $\pi^B_{**}(HC) \cong \Ext_C^{**}(k, k)$ and $\pi^B_{**}(HC \otimes HC) \cong HC_{**} \otimes (B \cotensor_C k)$ as left $\pi^B_{**}(HC)$-modules. It is claimed in \cite{pal-stable} that this is a Hopf algebra, but the computation for the right unit map $\eta_R$ is incorrect. If one carefully traces through all of the shearing and change-of-rings isomorphisms, one sees that $\eta_R$ is the map obtained by converting the left coaction of $B \cotensor_C k$ on $HC_{**}$ from Remark 1.2.13 of \cite{pal-stable} into a right coaction via the conjugation map of $B \cotensor_C k$. As a result,
    \[\left(\pi^B_{**}(HC), \pi^B_{**}(HC \otimes HC)\right) \cong \left(\Ext_C^{**}(k, k), \Ext_C^{**}(k, k) \otimes (B \cotensor_C k)\right)\]
is a split Hopf algebroid coming from the right coaction of the Hopf algebra $B \cotensor_C k$ on $\Ext_C^{**}(k, k)$. Therefore, any mention in \cite{pal-stable} of invariance under the coaction of the Hopf algebra $B \cotensor_C \FF_2$ can be equivalently interpreted as invariance under the right unit map $\eta_R$ of the Hopf algebroid $(\pi^B_{**}(HC), \pi^B_{**}(HC \otimes HC))$.

We will write $\Specinv(\Ext_B^{**}(\FF_2, \FF_2))$ for the subspace of $\Spec^h(\Ext_B^{**}(\FF_2, \FF_2))$ consisting of the invariant prime ideals (see Appendix A for precise definitions).

\subsection{The Steenrod algebra}

Let $A$ denote the mod 2 dual Steenrod algebra, which is a graded connected commutative Hopf algebra over $\FF_2$. Milnor showed in \cite{milnor-steenrod} that as an algebra,
    \[A \cong \FF_2[\xi_1, \xi_2, \xi_3, \ldots]\]
with $|\xi_n| = 2^n-1$, and the coproduct is given by
    \[\Delta(\xi_n) = \sum_{i=0}^n \xi_{n-i}^{2^i} \otimes \xi_i,\]
where $\xi_0 = 1$ by convention.

Adams and Margolis showed in \cite{adams-margolis-sub} that the quotient Hopf algebras of $A$ are exactly those of the form
    \[A/(\xi_1^{2^{n_1}}, \xi_2^{2^{n_2}}, \xi_3^{2^{n_3}}, \ldots).\]
for some sequence $(n_1, n_2, n_3, \ldots)$ in $\ZZ_{\geq 0} \cup \{\infty\}$ (where $\xi_i^{2^\infty}$ is interpreted as $0$) satisfying the following condition: for each $i, j$, either $n_i \leq n_{i+j} + j$ or $n_j \leq n_{i+j}$. The sequence $(n_1, n_2, n_3, \ldots)$ is referred to as the \emph{profile function} for the given quotient Hopf algebra.

Recall that a graded connected commutative Hopf algebra over a field $k$ of characteristic $p$ is called \emph{elementary} if its graded dual is commutative and satisfies $z^p = 0$ for all $z$ in the augmentation ideal. The elementary quotient Hopf algebras of $A$ were classified by Lin in \cite{lin-cohomology}, and they are precisely the quotient Hopf algebras of the following maximal elementary quotient Hopf algebras:
    \[E(m) = A/(\xi_1, \ldots, \xi_m, \xi_{m+1}^{2^{m+1}}, \xi_{m+2}^{2^{m+1}}, \xi_{m+3}^{2^{m+1}}, \ldots), \quad m \geq 0.\]

    \begin{figure}[t]
        \begin{subfigure}{0.45\textwidth}
        \begin{tikzpicture}[scale=0.6, every node/.style={transform shape}]
                \draw[stealth-stealth] (0,6.75) -- (0,0) -- (8.75,0);
                \draw[thick, -stealth] (0,0) -- (3,0) -- (3,3) -- (8.75,3);
                \draw (3.75,0.5) node[scale=1.5]{$\xi_{m+1}$};
                \draw (3.75,1.5) node[scale=1.5]{$\vdots$};
                \draw (3.75,2.5) node[scale=1.5]{$\xi_{m+1}^{2^m}$};
                \draw (5.25,0.5) node[scale=1.5]{$\xi_{m+2}$};
                \draw (5.25,1.5) node[scale=1.5]{$\vdots$};
                \draw (5.25,2.5) node[scale=1.5]{$\xi_{m+2}^{2^m}$};
                \draw (6.75,0.5) node[scale=1.5]{$\xi_{m+3}$};
                \draw (6.75,1.5) node[scale=1.5]{$\vdots$};
                \draw (6.75,2.5) node[scale=1.5]{$\xi_{m+3}^{2^m}$};
                \draw (8.25,0.5) node[scale=1.5]{$\cdots$};
                \draw (8.25,1.5) node[scale=1.5]{$\cdots$};
                \draw (8.25,2.5) node[scale=1.5]{$\cdots$};
        \end{tikzpicture}
        \caption{Profile function for $E(m)$}
        \label{fig:profile-E(m)}
        \end{subfigure}
    \quad
        \begin{subfigure}{0.45\textwidth}
        \begin{tikzpicture}[scale=0.6, every node/.style={transform shape}]
                \draw[stealth-stealth] (0,6.75) -- (0,0) -- (7.25,0);
                \draw[thick, -stealth] (0,0) -- (0,1) -- (1,1) -- (1,2) -- (2,2) -- (2,3) -- (3,3) -- (3,4) -- (4,4) -- (4,5) -- (5,5) -- (5,6) -- (6.25,6);
                \draw (0.625,0.5) node[scale=1.5]{$\xi_1$};
                \draw (1.625,0.5) node[scale=1.5]{$\xi_2$};
                \draw (1.625,1.5) node[scale=1.5]{$\xi_2^2$};
                \draw (2.625,0.5) node[scale=1.5]{$\xi_3$};
                \draw (2.625,1.5) node[scale=1.5]{$\xi_3^2$};
                \draw (2.625,2.5) node[scale=1.5]{$\xi_3^4$};
                \draw (3.625,0.5) node[scale=1.5]{$\xi_4$};
                \draw (3.625,1.5) node[scale=1.5]{$\xi_4^2$};
                \draw (3.625,2.5) node[scale=1.5]{$\xi_4^4$};
                \draw (3.625,3.5) node[scale=1.5]{$\xi_4^8$};
                \draw (4.625,0.5) node[scale=1.5]{$\xi_5$};
                \draw (4.625,1.5) node[scale=1.5]{$\xi_5^2$};
                \draw (4.625,2.5) node[scale=1.5]{$\xi_5^4$};
                \draw (4.625,3.5) node[scale=1.5]{$\xi_5^8$};
                \draw (4.625,4.5) node[scale=1.5]{$\xi_5^{16}$};
                \draw (5.625,0.5) node[scale=1.5]{$\xi_6$};
                \draw (5.625,1.5) node[scale=1.5]{$\xi_6^2$};
                \draw (5.625,2.5) node[scale=1.5]{$\xi_6^4$};
                \draw (5.625,3.5) node[scale=1.5]{$\xi_6^8$};
                \draw (5.625,4.5) node[scale=1.5]{$\xi_6^{16}$};
                \draw (5.625,5.5) node[scale=1.5]{$\xi_6^{32}$};
                \draw (6.625,0.5) node[scale=1.5]{$\cdots$};
                \draw (6.625,1.5) node[scale=1.5]{$\cdots$};
                \draw (6.625,2.5) node[scale=1.5]{$\cdots$};
                \draw (6.625,3.5) node[scale=1.5]{$\cdots$};
                \draw (6.625,4.5) node[scale=1.5]{$\cdots$};
                \draw (6.625,5.5) node[scale=1.5]{$\cdots$};
        \end{tikzpicture}
        \caption{Profile function for $D$}
        \label{fig:profile-D}
        \end{subfigure}
        \caption{Profile functions for quotient Hopf algebras of $A$}
    \end{figure}
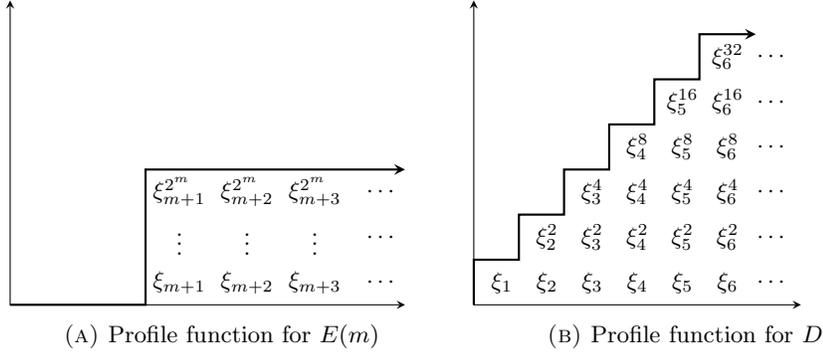

The quotient Hopf algebra $D$ mentioned in the Introduction is given by
    \[D = A/(\xi_1^2, \xi_2^4, \xi_3^8, \ldots, \xi_n^{2^n}, \ldots)\]
Note that $D$ is the smallest quotient Hopf algebra of $A$ containing all of the elementary quotients, which seems to give some explanation for its importance.

While the ring $\pi^A_{**}(HD) \cong \Ext_D^{**}(\FF_2, \FF_2)$ is not fully known, it is known ``up to nilpotents.'' To be more precise, let $\mathcal{Q}$ denote the category of elementary quotient Hopf algebras of $A$ with morphisms given by quotient maps. Then the quotient maps $D \twoheadrightarrow E$ for $E \in \mathcal{Q}$ induce a map
    \[\Ext_D^{**}(\FF_2, \FF_2) \to \lim_{E \in \mathcal{Q}} \Ext_E^{**}(\FF_2, \FF_2),\]
which by Proposition 4.2 of \cite{pal-quillen} is an \emph{$F$-isomorphism}, i.e., the kernel consists of nilpotent elements, and for every $y$ in the codomain, there is an integer $n$ such that $y^{2^n}$ is in the image. Palmieri explicitly computed
    \[\lim_{E \in \mathcal{Q}} \Ext_E^{**}(\FF_2, \FF_2) \cong \FF_2[h_{ts} : s < t]/(h_{ts}h_{vu} : t \leq u)\]
and an $A \cotensor_D \FF_2$-coaction given by
    \[h_{ts} \mapsto \sum_{j=0}^{\lfloor\frac{-s+t-1}{2}\rfloor} \sum_{i=j+s+1}^{t-j} \zeta_j^{2^s} \xi_{t-i-j}^{2^{i+j+s}} \otimes h_{i,j+s},\]
where $\zeta_j$ denotes the conjugate of $\xi_j$.

Palmieri's nilpotence theorem says that the object $HD$ detects nilpotence in $\Stable(A)$, suggesting that $HD$ plays a role analogous to $MU$ in stable homotopy theory. Therefore, just as the Balmer spectrum for the stable homotopy category of finite spectra is $\Spc(\text{SHC}^c) \cong \Specinv(MU_*)$, we might predict that $\Spc(\Stable(A)^c) \cong \Specinv(\Ext_D^{**}(\FF_2, \FF_2))$. This analogy also suggests a possible route for computing $\Spc(\Stable(A)^c)$: construct residue field objects $K(\p)$ analogous to the Morava $K$-theories for each $\p \in \Specinv(\Ext_D^{**}(\FF_2, \FF_2))$ and prove that they detect nilpotence. Using techniques from \cite{johnson-yosimura}, such a $K(\p)$ can be constructed for every $\p$ that is generated by an invariant regular sequence. Unfortunately, the author does not know whether or not every invariant prime ideal in $\Ext_D^{**}(\FF_2, \FF_2)$ can be generated by an invariant regular sequence.

From the perspective of modular representation theory, at first glance it might seem that we should expect the Balmer spectrum of $\Stable(A)^c$ to be homeomorphic to $\Spec^h(\Ext_A^{**}(\FF_2, \FF_2)$. In fact, Hovey and Palmieri used the approach of Benson, Carlson, and Rickard to prove this result for \emph{finite-dimensional} quotient Hopf algebras of $A$ (see Theorem \ref{thm:hovpal}). However, since $A$ is infinite-dimensional, some of the key tools and techniques (e.g., rank varieties) become unavailable. To see why we might instead expect $\Spc(\Stable(A)^c) \cong \Specinv(\Ext_D^{**}(\FF_2, \FF_2))$ from the modular representation theory perspective, consider Palmieri's Quillen stratification theorem, which says that there is an $F$-isomorphism
    \[\Ext_A^{**}(\FF_2, \FF_2) \to \Ext_D^{**}(\FF_2, \FF_2)^{A \cotensor_D \FF_2}.\]
Since $A \cotensor_D \FF_2$ is not a finite group but rather an infinite-dimensional group scheme, instead of looking at
    \[\Spec^h(\Ext_A^{**}(\FF_2, \FF_2)) \cong \Spec^h(\Ext_D^{**}(\FF_2, \FF_2)^{A \cotensor_D \FF_2}),\]
it seems to make more sense to look at the quotient stack
    \[[\Spec^h(\Ext_D^{**}(\FF_2, \FF_2))/\Spec^h(A \cotensor_D \FF_2)],\]
whose underlying topological space should be $\Specinv(\Ext_D^{**}(\FF_2, \FF_2))$. This is not entirely rigorous, but it serves as good motivation.


\section{Finite-dimensional quotient Hopf algebras of $A$}

In \cite{hov-pal-stable} and \cite{hov-pal-galois}, Hovey and Palmieri used techniques and tools from modular representation theory to classify the thick subcategories of $\Stable(B)^c$ for $B$ a finite-dimensional quotient Hopf algebra of $A$. In this brief section, we translate their result into the language of tensor triangular geometry.

\begin{thm}[Hovey-Palmieri]\label{thm:hovpal}
    For any finite-dimensional quotient Hopf algebra $B$ of $A$, the comparison map
        \[\rho_B : \Spc(\Stable(B)^c) \to \Spec^h(\Ext_B^{**}(\FF_2, \FF_2))\]
    is a homeomorphism.
\end{thm}

\begin{proof}
    By \cite{hov-pal-stable} and \cite{hov-pal-galois}, the support theory on $\Stable(B)^c$ given by
        \[\sigma_B(X) = \{\p \in \Spec^h(\Ext_B^{**}(\FF_2, \FF_2)) : X_\p \neq 0\}\]
    induces a bijection
        \[\left\{\begin{tabular}{@{}c@{}} thick subcategories \\ of $\Stable(B)^c$ \end{tabular}\right\} \stackrel{\sim}{\longleftrightarrow} \left\{\begin{tabular}{@{}c@{}} Thomason subsets \\ of $\Spec^h(\Ext_B^{**}(\FF_2, \FF_2))$ \end{tabular}\right\}.\]
    Therefore, by Theorem 16 of \cite{balmer-ttg}, the universal map
        \[\varphi_B : \Spec^h(\Ext_B^{**}(\FF_2, \FF_2)) \to \Spc(\Stable(B)^c)\]
    given by
        \begin{align*}
            \varphi_B(\p) &= \{X \in \Stable(B)^c : \p \notin \sigma_B(X)\} \\
            &= \{X \in \Stable(B)^c : X_\p = 0\}
        \end{align*}
    is a homeomorphism. Since
        \begin{align*}
            \rho_B(\varphi_B(\p)) &= \{f \in \Ext_B^{**}(\FF_2, \FF_2) : \cone(f) \notin \varphi_B(\p)\} \\
            &= \{f \in \Ext_B^{**}(\FF_2, \FF_2) : \cone(f)_\p \neq 0\} \\
            &= \{f \in \Ext_B^{**}(\FF_2, \FF_2) : f \text{ is not a unit in } \Ext_B^{**}(\FF_2, \FF_2)_\p\} \\
            &= \p,
        \end{align*}
    $\rho_B$ is the inverse of $\varphi_B$ and is hence also a homeomorphism.
\end{proof}


\section{Elementary quotient Hopf algebras of $A$}

Throughout this section, let $E$ denote an elementary quotient Hopf algebra of $A$. Let $m = \min\{n : \xi_{n+1} \in E\}$, and consider the finite-dimensional quotient Hopf algebras
    \[E_i = E/(\xi_{m+i+1}, \xi_{m+i+2}, \ldots)\]
for $i \geq 0$. Pictorially, $E_i$ consists of the first $i$ columns of $E$.

    \begin{figure}[h]
        \centering
        \begin{tikzpicture}[scale=0.75, every node/.style={transform shape}]
                    \draw[stealth-stealth] (0,5) -- (0,0) -- (13.25,0);
                    \draw[thick] (0,0) -- (3,0) -- (3,3) -- (7.675,3) -- (7.675,0) -- (13.25,0);
                    \draw[thick, -stealth, dashed] (0,0) -- (3,0) -- (3,2.99) -- (13.25,2.99);
                    \draw (3.875,0.5) node[scale=1.5]{$\xi_{m+1}$};
                    \draw (3.875,2) node[scale=1.5]{$\vdots$};
                    \draw (5.25,0.5) node[scale=1.5]{$\cdots$};
                    \draw (5.25,2) node[scale=1.5]{$\cdots$};
                    \draw (6.625,0.5) node[scale=1.5]{$\xi_{m+i}$};
                    \draw (6.625,2) node[scale=1.5]{$\vdots$};
                    \draw (8.875,0.5) node[scale=1.5]{$\xi_{m+i+1}$};
                    \draw (8.875,2) node[scale=1.5]{$\vdots$};
                    \draw (10.875,0.5) node[scale=1.5]{$\xi_{m+i+2}$};
                    \draw (10.875,2) node[scale=1.5]{$\vdots$};
                    \draw (12.675,0.5) node[scale=1.5]{$\cdots$};
                    \draw (12.675,2) node[scale=1.5]{$\cdots$};
            \end{tikzpicture}
        \caption{Profile functions for $E_i$ (solid line) and $E$ (dashed line)}
        \label{fig:enter-label}
    \end{figure}
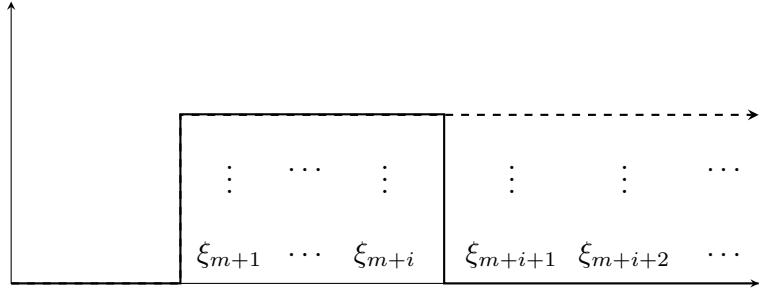

The aim of this section is to compute $\Spc(\Stable(E)^c)$ from $\Spc(\Stable(E_i)^c)$ using the continuity of the Balmer spectrum, a technique described by Gallauer in \cite{gal-filtered}. We will exploit the particularly nice structure of the elementary quotient Hopf algebras of $A$.

Recall the following definition from \cite{hps-axiomatic}.

\begin{defn}
    A \emph{stable morphism} is a tt-functor which admits a right adjoint and preserves compactness.
\end{defn}

For our interests, stable morphisms arise in the following way: a map $B \to B'$ of Hopf algebras induces a stable morphism $\Stable(B) \to \Stable(B')$ by Theorem 9.5.1 of \cite{hps-axiomatic}.

In particular, the quotient maps $E \twoheadrightarrow E_i$ and $E_j \twoheadrightarrow E_i$ for $i \leq j$ induce stable morphisms
    \[\res_i : \Stable(E) \to \Stable(E_i),\]
    \[\res_{j,i} : \Stable(E_j) \to \Stable(E_i).\]
which satisfy
    \[\res_{j,i} \circ \res_{k,j} = \res_{k,i},\]
    \[\res_{j,i} \circ \res_j = \res_i\]
for $i \leq j \leq k$. Notice that the Hopf algebra kernel of the quotient map $E \twoheadrightarrow E_i$ is isomorphic to the quotient Hopf algebra $E/(\xi_{m+1}, \ldots, \xi_{m+i})$, and similarly for $i \leq j$, the Hopf algebra kernel of the quotient map $E_j \twoheadrightarrow E_i$ is isomorphic to the quotient Hopf algebra $E_j/(\xi_{m+1}, \ldots, \xi_{m+i})$. Therefore, we get splittings $E_i \hookrightarrow E$ and $E_i \hookrightarrow E_j$
which induce stable morphisms
    \[\incl_i : \Stable(E_i) \to \Stable(E),\]
    \[\incl_{i,j} : \Stable(E_i) \to \Stable(E_j).\]
Note that for $i \leq j \leq k$,
    \[\incl_{j,k} \circ \incl_{i,j} \cong \incl_{i,k},\]
    \[\incl_j \circ \incl_{i,j} \cong \incl_i.\]
Moreover, for $i \leq j$,
    \[\res_{j,i} \circ \incl_{i,j} \cong \id_{\Stable(E_i)},\]
    \[\res_i \circ \incl_i \cong \id_{\Stable(E_i)}.\]

We may now assemble all of the inclusion functors into a pseudo-functor
    \[\Stable(E_\bullet)^c : \ZZ_{\geq 0} \to 2\text{-}ttCat\]
and a pseudo-natural transformation
    \[\incl : \Stable(E_\bullet)^c \to \Stable(E)^c.\]
Our goal is to show the following:
    \begin{itemize}
        \item $\incl : \Stable(E_\bullet)^c \to \Stable(E)^c$ is \emph{surjective on morphisms}, i.e., for each morphism $f : X \to Y$ in $\Stable(E)^c$, there exists $i \geq 0$ and a morphism $f_i : X_i \to Y_i$ in $\Stable(E_i)^c$ such that $\incl_i(f_i) = f$,
        \item $\incl : \Stable(E_\bullet)^c \to \Stable(E)^c$ \emph{detects isomorphisms}, i.e., if $X_i, Y_i \in \Stable(E_i)^c$ with $\incl_i(X_i) \cong \incl_i(Y_i)$ in $\Stable(E)^c$, then there exists $j \geq i$ such that $\incl_{i,j}(X_i) \cong \incl_{i,j}(Y_i)$ in $\Stable(E_j)^c$.
    \end{itemize}

We start by proving a helpful lemma.

\begin{lem}\label{lem:incl-of-res}
    Suppose $X \in \Stable(E)^c$ with $X \cong \incl_{i_0}(X_{i_0})$ for some $i_0$ and $X_{i_0} \in \Stable(E_{i_0})^c$. Then for $i \geq i_0$,
        \[\incl_i(\res_i(X)) \cong X\]
    and for $j \geq i \geq i_0$,
        \[\incl_{i,j}(\res_i(X)) \cong \res_j(X).\]
\end{lem}

\begin{proof}
    Since $\res_i \circ \incl_i \cong \id_{\Stable(E_i)}$ and $\incl_i \circ \incl_{i_0,i} \cong \incl_{i_0}$,
        \[\incl_{i_0,i}(X_{i_0}) \cong \res_i(\incl_i(\incl_{i_0,i}(X_{i_0}))) \cong \res_i(\incl_{i_0}(X_{i_0})) \cong \res_i(X),\]
    hence
        \[\incl_i(\res_i(X)) \cong \incl_i(\incl_{i_0,i}(X_{i_0})) \cong \incl_{i_0}(X_{i_0}) \cong X.\]
    Similarly, since $\incl_{i,j} \circ \incl_{i_0,i} \cong \incl_{i_0,j}$,
        \[\incl_{i,j}(\res_i(X)) \cong \incl_{i,j}(\incl_{i_0,i}(X_{i_0})) \cong \incl_{i_0,j}(X_{i_0}) \cong \res_{j}(X),\]
    where the last isomorphism comes from replacing $i$ with $j$ in the first step of the proof.
\end{proof}

\begin{lem}\label{lem:surjective on morphisms}
    The pseudo-natural transformation
        \[\incl : \Stable(E_\bullet)^c \to \Stable(E)^c\]
    is surjective on morphisms. 
\end{lem}

\begin{proof}
    Consider the subcategory $\mathcal C$ of $\Stable(E)^c$ consisting of objects $X$ for which $X \cong \incl_i(X_i)$ for some $i \geq 0$ and some $X_i \in \Stable(E_i)^c$ and morphisms $X \xrightarrow{f} Y$ for which $f \cong \incl_i(f_i)$ for some $i \geq 0$ and some morphism $X_i \xrightarrow{f_i} Y_i$ in $\Stable(E_i)^c$. We will show that $\mathcal{C}$ is a thick subcategory; since $\mathcal{C}$ contains the unit object $S^0_E$, this will imply that $\mathcal{C} = \Stable(E)^c$, i.e., $\incl : \Stable(E_\bullet)^c \to \Stable(E)^c$ is surjective on morphisms.

    First, let us show that $\mathcal C$ is a full subcategory. Suppose $X \xrightarrow{f} Y$ is a morphism with $X, Y \in \mathcal C$. Choose $i_0$ such that $X \cong \incl_{i_0}(X_{i_0})$ and $Y \cong \incl_{i_0}(Y_{i_0})$. By Lemma \ref{lem:incl-of-res}, for $j \geq i \geq i_0$ we have 
        \[\incl_{i,j}(\res_i(X)) \cong \res_j(X), \quad \incl_{i,j}(\res_i(Y)) \cong \res_j(Y),\]
    so the inclusion functors induce a direct system with maps given by
        \[\left[\res_i(X), \res_i(Y)\right]^{E_i}_{**} \xrightarrow{\incl_{i,j}} \left[\res_j(X), \res_j(Y)\right]^{E_j}_{**}.\]
    Moreover, since
        \[\incl_i(\res_i(X)) \cong X, \quad \incl_i(\res_i(Y)) \cong Y\]
    for $i \geq i_0$ by Lemma \ref{lem:incl-of-res}, we get maps
        \[\left[\res_i(X), \res_i(Y)\right]^{E_i}_{**} \xrightarrow{\incl_i} [X, Y]^E_{**}\]
    which induce a map
        \[\dirlim_{i \geq i_0} \left[\res_i(X), \res_i(Y)\right]^{E_i}_{**} \to [X, Y]^E_{**}.\]
    Notice that the maps in the direct system can be identified as follows
        \begin{center}
            \begin{tikzcd}
                \left[\res_i(X), \res_i(Y)\right]^{E_i}_{**} \ar[r, "\incl_{i,j}"] \ar[d, "\cong"'] & \left[\res_j(X), \res_j(Y)\right]^{E_j}_{**} \ar[d, "\cong"] \\
                \left[X, E \cotensor_{E_i} \res_i(Y)\right]^E_{**} \ar[d, "\cong"'] & \left[X, E \cotensor_{E_j} \res_j(Y)\right]^E_{**} \ar[d, "\cong"] \\
                \left[X, (E \cotensor_{E_i} \FF_2) \otimes Y\right]^E_{**} \ar[r] & \left[X, (E \cotensor_{E_j} \FF_2) \otimes Y\right]^E_{**},
            \end{tikzcd}
        \end{center}
    where the bottom map is induced by the projection $E \cotensor_{E_i} \FF_2 \to E \cotensor_{E_j} \FF_2$. Then we also get the following identification
        \begin{center}
            \begin{tikzcd}
                \left[\res_i(X), \res_i(Y)\right]^{E_i}_{**} \ar[d, "\cong"'] \ar[r, "\incl_i"] & \left[X, Y\right]^E_{**}, \\
                \left[X, (E \cotensor_{E_i} \FF_2) \otimes Y\right]^E_{**} \ar[ru]
            \end{tikzcd}
        \end{center}
    where the diagonal map is induced by the projection $E \cotensor_{E_i} \FF_2 \to E \cotensor_E \FF_2 \cong \FF_2$. Now since $X$ is compact, we have
        \begin{center}
            \begin{tikzcd}
                \displaystyle\dirlim_{i \geq i_0} \left[\res_i(X), \res_i(Y)\right]^{E_i}_{**} \ar[dd] \ar[r, "\cong"] & \displaystyle\dirlim_{i \geq i_0} \left[X, (E \cotensor_{E_i} \FF_2) \otimes Y\right]^E_{**} \ar[d, "\cong"] \\
                & \left[X, \displaystyle\dirlim_{i \geq i_0} \left((E \cotensor_{E_i} \FF_2) \otimes Y\right)\right]^E_{**} \ar[d, "\cong"] \\
                \left[X, Y\right]^E_{**} & \left[X, \displaystyle\dirlim_{i \geq i_0} (E \cotensor_{E_i} \FF_2) \otimes Y\right]^E_{**} \ar[l, "\cong"].
            \end{tikzcd}
        \end{center}
    Therefore, since $f \in [X, Y]^E_{**}$, we have $f \cong \incl_{i_1}(f_{i_1})$ for some $i_1 \geq i_0$ and some morphism $\res_{i_1}(X) \xrightarrow{f_{i_1}} \res_{i_1}(Y)$ in $\Stable(E_{i_1})^c$. Thus, $\mathcal{C}$ is a full subcategory.

    Now suppose we have a cofiber sequence $X \xrightarrow{f} Y \to Z$ with $X, Y \in \mathcal C$. Since $\mathcal C$ is full, we know that $f \cong \incl_i(f_i)$ for some morphism $X_i \xrightarrow{f_i} Y_i$ in $\Stable(E_i)^c$. Let $Z_i$ be the cofiber of $X_i \xrightarrow{f_i} Y_i$.
    Since $\incl_i$ is exact, we get that
        \[\incl_i(Z_i) \cong \incl_i(\cone(f_i)) \cong \cone(\incl_i(f_i)) \cong \cone(f) \cong Z.\]

    Lastly, suppose that $Z = X \oplus Y \in \mathcal C$. Consider the cochain map $Z \xrightarrow{\a} Z$ given by $\a(x \oplus y) = 0 \oplus y$. In $\Ch(E)$, the kernel of this map is $X$. Now since $\mathcal C$ is a full subcategory of $\Stable(E)^c$, choose $i$ such that $\a \cong \incl_i(\a_i)$ for some morphism $Z_i \xrightarrow{\a_i} Z_i$ in $\Stable(E_i)^c$. Observe that $\res_i(\a) \cong \res_i(\incl_i(\a_i)) \cong \a_i$, hence we may replace $\a_i$ with $\res_i(\a)$ so that we actually have equality $\res_i(\a) = \a_i$. This guarantees that $\a_i$ is idempotent as an endomorphism of $\res_i(Z)$ in $\Ch(E_i)$, therefore the kernel $X_i = \ker(\a_i)$ in $\Ch(E_i)$ is a cochain complex of injectives. In particular, $X_i$ may also be viewed as an object in $\Stable(E_i)^c$. Since the functor $\incl_i : \Ch(E_i) \to \Ch(E)$ is exact, we get that $\incl_i(X_i) = \incl_i(\ker(\a_i)) = \ker(\incl_i(\a_i))$. Therefore, in $\Stable(E)^c$, we have $\incl_i(X_i) \cong \ker(\incl_i(\alpha_i)) \cong \ker(\a) \cong X$. Since we have a cofiber sequence $X \to Z \to Y$, by the previous paragraph we also have $Y \in \mathcal C$. This completes the proof.
\end{proof}

\begin{lem}\label{lem:detects isomorphisms}
    The pseudo-natural transformation
        \[\incl : \Stable(E_\bullet)^c \to \Stable(E)^c\]
    detects isomorphisms.
\end{lem}

\begin{proof}
    As mentioned in \cite{gal-filtered}, detecting isomorphisms is equivalent to detecting zero objects. This follows from $\res_i \circ \incl_i \cong \id_{\Stable(E_i)}$: if $X_i \in \Stable(E_i)^c$ with $\incl_i(X_i) \cong 0$, then
        \[X_i \cong \res_i(\incl_i(X_i)) \cong \res_i(0) \cong 0.\]
\end{proof}

The next result now follows immediately by Proposition 8.5 of \cite{gal-filtered}.

\begin{thm}\label{thm:E-continuity}
    The induced map
        \[\invlim_i \Spc(\incl_i) : \Spc(\Stable(E)^c) \to \invlim_{i} \Spc(\Stable(E_i)^c)\]
    is a homeomorphism.
\end{thm}

Using Hovey and Palmieri's result (Theorem \ref{thm:hovpal}) for $\Spc(\Stable(E_i)^c)$, we can now explicitly compute $\Spc(\Stable(E)^c)$.

\begin{cor}\label{cor:Spc(Stable(E)^c)}
    The comparison map
        \[\rho_E : \Spc(\Stable(E)^c) \to \Spec^h(\Ext_E^{**}(\FF_2, \FF_2))\]
    is a homeomorphism.
\end{cor}

\begin{proof}
    By the naturality of comparison maps, we have a commutative diagram
        \begin{center}
            \begin{tikzcd}[column sep = large]
                \Spc(\Stable(E)^c) \ar[d, "\invlim_i \Spc(\incl_i)"'] \ar[r, "\rho_E"] & \Spec^h(\Ext_E^{**}(\FF_2, \FF_2)) \ar[d] \\
                \displaystyle\invlim_i \Spc(\Stable(E_i)^c) \ar[r, "\invlim_i \rho_{E_i}"] & \displaystyle\invlim_i \Spec^h(\Ext_{E_i}^{**}(\FF_2, \FF_2)).
            \end{tikzcd}
        \end{center}
    Note that $\invlim_i \Spc(\incl_i)$ is a homeomorphism by Theorem \ref{thm:E-continuity}, and $\invlim_i \rho_{E_i}$ is a homeomorphism by Theorem \ref{thm:hovpal}. Moreover, note that
        \[\Ext_{E_i}^{**}(\FF_2, \FF_2) \cong \FF_2[h_{ts} : \xi_t^{2^s} \neq 0 \in E_i], \quad \Ext_E^{**}(\FF_2, \FF_2) \cong \FF_2[h_{ts} : \xi_t^{2^s} \neq 0 \in E],\]
    and all of the maps induced by the various inclusions are just inclusions of polynomial rings. Therefore, the map $\dirlim_i \Ext_{E_i}^{**}(\FF_2, \FF_2) \to \Ext_E^{**}(\FF_2, \FF_2)$ is an isomorphism and hence the vertical map on the right-hand-side of the diagram is also a homeomorphism. Thus, $\rho_E : \Spc(\Stable(E)^c) \to \Spec^h(\Ext_E^{**}(\FF_2, \FF_2))$ is a homeomorphism as well.
\end{proof}

\begin{rem}\label{rem:E-supp}
    It will be necessary in the next section for us to understand what supports look like under the homeomorphism of Corollary \ref{cor:Spc(Stable(E)^c)}.
    By Proposition 2.10 of \cite{lau-stacks}, for any $X \in \Stable(E)^c$ we have
        \[\rho_E(\supp_E(X)) = \{\p \in \Spec^h(\Ext_E^{**}(\FF_2, \FF_2)) : X_\p \neq 0\},\]
    where $\supp_E$ denotes the universal support theory on $\Stable(E)^c$.
\end{rem}


\section{The quotient Hopf algebra $D$}

Recall the quotient Hopf algebra
    \[D = A/(\xi_1^2, \xi_2^4, \xi_3^8, \ldots, \xi_n^{2^n}, \ldots).\]
Let $\mathcal{Q}$ denote the category of elementary quotient Hopf algebras of $A$ with morphisms given by quotient maps. For each $E \in \mathcal{Q}$, the quotient map $D \twoheadrightarrow E$ induces a stable morphism
    \[\res_{D,E} : \Stable(D) \to \Stable(E).\]
These functors give rise to a map
    \[\colim_{E \in \mathcal{Q}} \Spc(\res_{D,E}) : \colim_{E \in \mathcal{Q}} \Spc(\Stable(E)^c) \to \Spc(\Stable(D)^c).\]
Define another map
    \[\varphi_D : \Spec^h(\Ext_D^{**}(\FF_2, \FF_2)) \to \Spc(\Stable(D)^c)\]
via the following diagram
    \begin{center}
        \begin{tikzcd}[column sep = huge]
            \colim_{E \in \mathcal{Q}} \Spc(\Stable(E)^c) \ar[d, "\cong"'] \ar[rr, "\colim_E \Spc(\res_{D,E})"] & & \Spc(\Stable(D)^c), \\
            \colim_{E \in \mathcal{Q}} \Spec^h(\Ext_E^{**}(\FF_2, \FF_2)) \ar[d, "\cong"'] \\
            \Spec^h(\Ext_D^{**}(\FF_2, \FF_2)) \ar[rruu, "\varphi_D"']
        \end{tikzcd}
    \end{center}
where the first vertical map is a homeomorphism by Corollary \ref{cor:Spc(Stable(E)^c)}, and the second is a homeomorphism by Proposition 4.2 of \cite{pal-quillen}.

The main goal of this section is to show that $\varphi_D$ is a homeomorphism.

\begin{prop}\label{prop:D-surj}
    The map
        \[\colim_{E \in \mathcal{Q}} \Spc(\res_{D,E}) : \colim_{E \in \mathcal{Q}} \Spc(\Stable(E)^c) \to \Spc(\Stable(D)^c)\]
    is surjective, therefore the map
        \[\varphi_D : \Spec^h(\Ext_D^{**}(\FF_2, \FF_2)) \to \Spc(\Stable(D)^c)\]
    is also surjective.
\end{prop}

\begin{proof}
    We have the following commutative diagram
        \begin{center}
            \begin{tikzcd}[column sep = large]
                \displaystyle\bigsqcup_{E \in \mathcal{Q}} \Spc(\Stable(E)^c) \ar[d] \ar[rrd, "\bigsqcup_E \Spc(\res_{D,E})"] & &  \\
            \displaystyle\colim_{E \in \mathcal{Q}} \Spc(\Stable(E)^c) \ar[rr, "\colim_E \Spc(\res_{D,E})"'] & & \Spc(\Stable(D)^c).
            \end{tikzcd}
        \end{center}
    The family $\{\res_{D,E} : \Stable(D) \to \Stable(E)\}_{E \in \mathcal{Q}}$ of stable morphisms jointly detects nilpotence by Theorem 5.1.6 of \cite{pal-stable} (replacing $A$ with $D$, which can be done as in \cite{pal-nilpotence-i}). Hence, by Theorem 1.3 of \cite{bchs-surjectivity}, the map $\bigsqcup_E \Spc(\res_{D,E})$ is surjective. Thus, $\colim_E \Spc(\res_{D,E})$ is also surjective.
\end{proof}

Let us abuse notation and also write $\res_{D,E} : \Ext_D^{**}(\FF_2, \FF_2) \to \Ext_E^{**}(\FF_2, \FF_2)$ for the induced map on Ext. We will write
    \[\res_{D,E}^* : \Spec^h(\Ext_E^{**}(\FF_2, \FF_2)) \to \Spec^h(\Ext_D^{**}(\FF_2, \FF_2))\]
for the induced map on $\Spec^h$.

The following result describes how algebraic localizations behave under restriction from $D$ to an elementary quotient Hopf algebra. Hovey and Palmieri proved a general result of this sort in Corollary 4.12 of \cite{hov-pal-stable}, but their proof relied on the assumption that the given Hopf algebras are finite-dimensional.

\begin{prop}\label{prop:res-of-loc}
    Suppose $X \in \Stable(D)^c$, $E \in \mathcal{Q}$, and $\p \in \Spec^h(\Ext_D^{**}(\FF_2, \FF_2))$. Then
        \[\res_{D,E}(X_\p) \cong \begin{cases} \res_{D,E}(X)_\q & \text{if } \res_{D,E}^*(\q) = \p, \\ 0 & \text{if } \p \notin \im(\res_{D,E}^*). \end{cases}\]
\end{prop}

\begin{proof}
    Let $T = \Ext_D^{**}(\FF_2, \FF_2) \ssm \p$ and let $\mathcal I$ denote the category whose objects are the elements of $T$ and whose morphisms are given by multiplication by elements of $T$. Consider the diagram $F : \mathcal{I} \to \Stable(D)^c$ which sends every object to (a suspension of) $X$ and each morphism corresponding to $t \in T \subseteq \pi^D_{**}(S^0_D)$ to the induced self-map on $X$. It is not too hard to see that $X_\p$ is the minimal weak colimit (in the sense of \cite{hps-axiomatic}) of $F$. Since $\res_{D,E}$ is a left adjoint, it preserves weak colimits. Hence, $\res_{D,E}(X_\p)$ is a weak colimit of the diagram $\res_{D,E} \circ F$. To see that $\res_{D,E}(X_\p)$ is indeed the minimal weak colimit of $\res_{D,E} \circ F$, by Proposition 2.2.2 of \cite{hps-axiomatic} we must show that the induced map
        \[\colim_\mathcal{I} \pi_{**}^E(\res_{D,E}(X)) \to \pi_{**}^E(\res_{D,E}(X_\p))\]
    is an isomorphism.
    Since $\pi^E_{**} \circ \res_{D,E}$ is a homology theory, this follows immediately from the definition of minimal weak colimit applied to $X_\p$. Therefore, we may compute $\res_{D,E}(X_\p)$ by computing the minimal weak colimit of the diagram $\res_{D,E} \circ F$, which is the algebraic localization of $\res_{D,E}(X)$ obtained by inverting the elements of the multiplicatively closed set $\res_{D,E}(T)$.

    First we consider the case that $\p \notin \im(\res_{D,E}^*)$. We have the following commutative diagram
        \begin{center}
            \begin{tikzcd}
                \Ext_D^{**}(\FF_2, \FF_2) \ar[r, "\res_{D,E}"] \ar[d, "\cong_F"'] & \Ext_E^{**}(\FF_2, \FF_2), \\
                \displaystyle\lim_{E' \in \mathcal{Q}} \Ext_{E'}^{**}(\FF_2, \FF_2) \ar[ru]
            \end{tikzcd}
        \end{center}
    where the vertical map is an $F$-isomorphism by Proposition 4.2 of \cite{pal-quillen}. Moreover, under the identifications
        \[\lim_{E' \in \mathcal{Q}} \Ext_{E'}^{**}(\FF_2, \FF_2) \cong \FF_2[h_{ts} : s < t]/(h_{ts}h_{vu} : u \geq t)\]
    and
        \[\Ext_E^{**}(\FF_2, \FF_2) \cong \FF_2[h_{ts} : \xi_t^{2^s} \neq 0 \in E]\]
    (the first of which comes from Theorem 1.3 of \cite{pal-quillen}), the diagonal map becomes the apparent quotient map which kills all of the $h_{ts}$'s for which $\xi_t^{2^s} \mapsto 0$ in $E$. Since the vertical $F$-isomorphism induces a homeomorphism on $\Spec^h$, consider the homogeneous prime ideal $\p'$ in $\lim_{E' \in \mathcal{Q}} \Ext_{E'}^{**}(\FF_2, \FF_2)$ corresponding to $\p$. Then $\p \notin \im(\res_{D,E}^*)$ means that $\p'$ does not contain the kernel of the quotient map, i.e., there is some $h_{ts}$ in the kernel of the quotient map with $h_{ts} \in \big(\lim_{E' \in \mathcal{Q}} \Ext_{E'}^{**}(\FF_2, \FF_2)\big) \ssm \p'$. Then for some $n$, there is an element $x \in \Ext_D^{**}(\FF_2, \FF_2)$ which maps to $h_{ts}^{2^n}$ under the $F$-isomorphism, therefore $x \in \Ext_D^{**}(\FF_2, \FF_2) \ssm \p$ and $\res_{D,E}(x) = 0$. But then the minimal weak colimit of the diagram $\res_{D,E} \circ F$ is an algebraic localization which inverts zero, thus $\res_{D,E}(X_\p) = 0$.
    
    Now we consider the case that $\p = \res_{D,E}^*(\q)$ for some $\q \in \Spec^h(\Ext_E^{**}(\FF_2, \FF_2))$. We have
        \[T = \Ext_D^{**}(\FF_2, \FF_2) \ssm \res_{D,E}^*(\q) = \res_{D,E}^{-1}(\Ext_E^{**}(\FF_2, \FF_2) \ssm \q),\]
    so $\res_{D,E}(T) \subseteq \Ext_E^{**}(\FF_2, \FF_2) \ssm \q$. Now suppose $y \in \Ext_E^{**}(\FF_2, \FF_2) \ssm \q$. Recalling the commutative diagram from above, since the diagonal map is surjective and the vertical map is an $F$-isomorphism, $y^{2^n} = \res_{D,E}(x)$ for some $x \in \Ext_D^{**}(\FF_2, \FF_2)$. Since $x \in \res_{D,E}^{-1}(\Ext_E^{**}(\FF_2, \FF_2) \ssm \q) = T$, we have $y^{2^n} \in \res_{D,E}(T)$. Therefore, inverting $\res_{D,E}(T)$ is equivalent to inverting $\Ext_E^{**}(\FF_2, \FF_2) \ssm \q$, thus $\res_{D,E}(X_\p) \cong \res_{D,E}(X)_\q$.
\end{proof}

We now define notions of support for $\Stable(D)^c$ and $\Stable(E)^c$ for $E \in \mathcal{Q}$.

\begin{defn}
    Given $X \in \Stable(D)^c$, define
        \begin{align*}
            \sigma_D(X) &= \left\{\p \in \Spec^h(\Ext_D^{**}(\FF_2, \FF_2)) : X_\p \neq 0\right\}.
        \end{align*}
    Similarly, for each $E \in \mathcal{Q}$, given $X \in \Stable(E)^c$, define
        \begin{align*}
            \sigma_E(X) &= \left\{\p \in \Spec^h(\Ext_E^{**}(\FF_2, \FF_2)) : X_\p \neq 0\right\} \\
            &= V\left(\ann_{\Ext_E^{**}(\FF_2, \FF_2)}\left(\pi^E_{**}(X)\right)\right).
        \end{align*}
\end{defn}

\begin{rem}
    Notice that for $E \in \mathcal{Q}$, the ring
        \[\pi^E_{**}(S^0_E) \cong \Ext_E^{**}(\FF_2, \FF_2) \cong \FF_2[h_{ts} : \xi_t^{2^s} \neq 0 \in E]\]
    is coherent, hence $\pi^E_{**}(X)$ is a finitely generated module over $\Ext_E^{**}(\FF_2, \FF_2)$ for $X \in \Stable(A)^c$ (e.g., by the methods in Section 1 of \cite{conner-smith-complex}). This justifies the second equality in our definition for $\sigma_E(X)$.
\end{rem}

The following lemma provides a sort of Quillen stratification for $\sigma_D$ in terms of the $\sigma_E$'s, which allows us to relate it to the universal support theory for $\Stable(D)^c$.

\begin{lem}\label{lem:quillen-supports}
    For any $X \in \Stable(D)^c$,
        \[\varphi_D^{-1}(\supp_D(X)) = \bigcup_{E \in \mathcal{Q}} \res_{D,E}^*(\sigma_E(\res_{D,E}(X))) = \sigma_D(X).\]
\end{lem}

\begin{proof}
    For the first equality, note that for each $E \in \mathcal{Q}$, the preimage of $\supp_D(X)$ under $\res_{D,E}^* : \Spc(\Stable(E)^c) \to \Spc(\Stable(D)^c)$ is $\supp_E(\res_{D,E}(X))$. Tracing through the homeomorphisms
        \[\colim_{E \in \mathcal{Q}} \Spc(\Stable(E)^c) \cong \colim_{E \in \mathcal{Q}} \Spec^h(\Ext_E^{**}(\FF_2, \FF_2)) \cong \Spec^h(\Ext_D^{**}(\FF_2, \FF_2)),\]
    we see that
        \begin{align*}
            \varphi_D^{-1}(\supp_D(X)) &= \bigcup_{E \in \mathcal{Q}} \res_{D,E}^*(\rho_E(\supp_E(\res_{D,E}(X)))) \\
            &= \bigcup_{E \in \mathcal{Q}} \res_{D,E}^*(\sigma_E(\res_{D,E}(X))),
        \end{align*}
    where we have used that $\rho_E \circ \supp_E = \sigma_E$ on $\Stable(E)^c$ by Remark \ref{rem:E-supp}.

    Now we show the second equality. Suppose $\p \in \bigcup_{E \in \mathcal{Q}} \res_{D,E}^*(\sigma_E(\res_{D,E}(X)))$. Then $\p = \res_{D,E}^*(\q)$ for some $E \in \mathcal{Q}$ and $\q \in \Spec^h(\Ext_E^{**}(\FF_2, \FF_2))$ such that $\res_{D,E}(X)_\q \neq 0$. By Lemma \ref{prop:res-of-loc},
        \[\res_{D,E}(X_\p) = \res_{D,E}(X)_\q \neq 0,\]
    hence $X_\p \neq 0$, i.e., $\p \in \sigma_D(X)$.

    Conversely, suppose $\p \notin \bigcup_{E \in \mathcal{Q}} \res_{D,E}^*(\sigma_E(\res_{D,E}(X)))$. Then $\res_{D,E}(X_\p) = 0$ and hence $\res_{D,E}(X_\p \otimes (DX)_\p) = 0$ for all $E \in \mathcal{Q}$, where $DX$ denotes the Spanier-Whitehead dual of $X$. By Theorem 5.1.6 of \cite{pal-stable} applied to $D$ instead of $A$, the family $\{\res_{D,E} : \Stable(D) \to \Stable(E)\}_{E \in \mathcal{Q}}$ detects nilpotence and therefore detects ring objects. Since $X_\p \otimes (DX)_\p = (X \otimes DX)_\p$ is a ring object, we have $X_\p \otimes (DX)_\p = 0$. But $X_\p$ is a retract of $(X \otimes DX \otimes X)_\p = X_\p \otimes (DX)_\p \otimes X_\p = 0$, hence $X_\p = 0$, i.e., $\p \notin \sigma_D(X)$.
\end{proof}

Now we show that every quasi-compact open subset of $\Spec^h(\Ext_D^{**}(\FF_2, \FF_2))$ arises as the $\sigma_D$-support of a compact object.

\begin{lem}\label{lem:D-koszul}
    For any finitely generated homogeneous ideal $I$ in $\Ext_D^{**}(\FF_2, \FF_2)$, there exists an object $X \in \Stable(D)^c$ with $\sigma_D(X) = V(I)$.
\end{lem}

\begin{proof}
    Let us write $I = (f_1, \ldots, f_k)$ with each $f_i$ homogeneous. Consider $X_0 = S^0_D$ and
        \[X_j = S^0_D/f_1 \otimes \cdots \otimes S^0_D/f_j\]
    for $j = 1, \ldots, k$. We will prove by induction on $j$ that $\sigma_D(X_j) = V((f_1, \ldots, f_j))$. Our base case is $\sigma_D(X_0) = \sigma_D(S^0_D) = V((0))$. Now assume that $\sigma_D(X_{j-1}) = V((f_1, \ldots, f_{j-1}))$. Tensoring the cofiber sequence
        \[S^0_D \xrightarrow{f_j} S^0_D \to S^0_D/f_j\]
    with $X_{j-1}$ gives
        \[X_{j-1} \xrightarrow{f_j} X_{j-1} \to X_j.\]
    We then see from the long exact sequence
        \begin{center}
            \begin{tikzcd}
                \cdots \ar[r] & \pi^D_{**}(X_{j-1}) \ar[r, "f_j"] & \pi^D_{**}(X_{j-1}) \ar[r] & \pi^D_{**}(X_j) \ar[r] & \cdots
            \end{tikzcd}
        \end{center}
    that $\p \in \sigma_D(X_j)$ if and only if $\p \in \sigma_D(X_{j-1})$ and $f_j$ is not a unit in $\Ext_D^{**}(\FF_2, \FF_2)_\p$, i.e., $f_j \in \p$. Hence,
        \begin{align*}
            \sigma_D(X_j) &= \sigma_D(X_{j-1}) \cap V((f_j)) \\
            &= V((f_1, \ldots, f_{j-1})) \cap V((f_j)) \\
            &= V((f_1, \ldots, f_{j-1}, f_j)),
        \end{align*}
    completing the induction step and finishing the proof by taking $X = X_k$.
\end{proof}

Next we prove a more general tt-geometry result, which we will use here for $\Stable(D)^c$ and in Section 6 for $\Stable(A)^c$.

\begin{prop}\label{prop:tt-result}
    Suppose $\mathcal{C}$ is an essentially small tt-category and $X$ is a spectral space with a continuous map $\varphi : X \to \Spc(\mathcal{C})$. If for every closed subset $Z \subseteq X$ with quasi-compact complement there exists an object $a \in \mathcal{C}$ with $\varphi^{-1}(\supp(a)) = Z$, then $\varphi$ is injective. If $\varphi$ is also surjective, then it is a homeomorphism.
\end{prop}

\begin{proof}
    Suppose $x \neq y$ in $X$. Since $X$ is spectral, without loss of generality there is a quasi-compact open subset $U \subseteq X$ such that $x \in U$ and $y \notin U$. Consider the complement $Z = X \ssm U$, which has $x \notin Z$ and $y \in Z$. By our hypothesis, there is an object $a \in \mathcal{C}$ with $\varphi^{-1}(\supp(a)) = Z$. But then $\varphi(x) \notin \supp(a)$ while $\varphi(y) \in \supp(a)$, so $\varphi(x) \neq \varphi(y)$. Thus, $\varphi$ is injective.

    Now suppose that $\varphi$ is also surjective. To show that $\varphi$ is a homeomorphism, it suffices to show that $\varphi$ is open. Since $X$ is spectral, its quasi-compact open subsets form a basis for the topology, hence it further suffices to show that $\varphi(U)$ is open for any quasi-compact open subset $U \subseteq X$. To this end, suppose $U \subseteq X$ is a quasi-compact open subset, and consider its complement $Z = X \ssm U$. By our hypothesis, there is an object $a \in \mathcal{C}$ with $\varphi^{-1}(\supp(a)) = Z$. Since $\varphi$ is bijective,
        \[\varphi(U) = \varphi(X \ssm Z) = \Spc(\mathcal{C}) \ssm \varphi(Z) = \Spc(\mathcal{C}) \ssm \supp(a),\]
    which is open in $\Spc(\mathcal{C})$. This finishes the proof.
\end{proof}

We now prove the main result of this section, which is a computation of the Balmer spectrum of $\Stable(D)^c$.

\begin{thm}\label{thm:Spc(Stable(D)^c)}
    The map
        \[\varphi_D : \Spec^h(\Ext_D^{**}(\FF_2, \FF_2)) \to \Spc(\Stable(D)^c)\]
    is a homeomorphism, hence the map
        \[\colim_{E \in \mathcal{Q}} \Spc(\res_{D,E}) : \colim_{E \in \mathcal{Q}} \Spc(\Stable(E)^c) \to \Spc(\Stable(D)^c)\]
    is also a homeomorphism.
\end{thm}

\begin{proof}
    If $Z \subseteq \Spec^h(\Ext_D^{**}(\FF_2, \FF_2))$ is a closed subset with quasi-compact complement, then we have
    $Z = V(I)$ for some finitely generated homogeneous ideal $I$ in $\Ext_D^{**}(\FF_2, \FF_2)$.
    By Lemmas \ref{lem:quillen-supports} and \ref{lem:D-koszul} there exists an object $X \in \Stable(D)^c$ with
        \[\varphi_D^{-1}(\supp_D(X)) = \sigma_D(X) = V(I) = Z.\]
    Moreover, $\varphi_D$ is surjective by Proposition \ref{prop:D-surj}. Thus, by Proposition \ref{prop:tt-result}, $\varphi_D$ is a homeomorphism.
\end{proof}

Note that the second homeomorphism of Theorem \ref{thm:Spc(Stable(D)^c)} can be seen as a sort of tt-categorification of Quillen stratification for $D$. By Corollary \ref{cor:Spc(Stable(E)^c)} and the naturality of comparison maps, we can conclude that the comparison map for $\Stable(D)^c$ is also a homeomorphism. In fact, we have the following result.

\begin{cor}
    The comparison map
        \[\rho_D : \Spc(\Stable(D)^c) \to \Spec^h(\Ext_D^{**}(\FF_2, \FF_2))\]
    is a homeomorphism which is inverse to $\varphi_D$.
\end{cor}

\begin{proof}
    It suffices to show that $\rho_D \circ \varphi_D = \id$, which follows from a similar computation to that given at the end of the proof of Theorem \ref{thm:hovpal}.
\end{proof}

\begin{rem}\label{rem:D-supp}
    Again by Proposition 2.10 of \cite{lau-stacks}, for any $X \in \Stable(D)^c$ we have
        \[\sigma_D(X) = \rho_D(\supp_D(X)),\]
    where $\supp_D$ denotes the universal support theory on $\Stable(D)^c$.
\end{rem}


\section{The dual Steenrod algebra $A$}

In this section we compute $\Spc(\Stable(A)^c)$, proving our main theorem. We will start by constructing a map from our candidate space $\Specinv(\Ext_D^{**}(\FF_2, \FF_2))$ to $\Spc(\Stable(A)^c)$, and then we will proceed to show that this map is a homeomorphism.

Recall that $\Specinv(\Ext_D^{**}(\FF_2, \FF_2))$ denotes the subspace of $\Spec^h(\Ext_D^{**}(\FF_2, \FF_2))$ consisting of the invariant prime ideals of the Hopf algebroid
    \[\left(\pi^A_{**}(HD), \pi^A_{**}(HD \otimes HD)\right) \cong \left(\Ext_D^{**}(\FF_2, \FF_2), \Ext_D^{**}(\FF_2, \FF_2) \otimes (A \cotensor_D \FF_2)\right).\]
Given a homogeneous ideal $I$, we write
    \[V^\textup{inv}(I) = V(I) \cap \Specinv(\Ext_D^{**}(\FF_2, \FF_2))\]
for the closed subset consisting of all invariant prime ideals containing $I$. We also write $I^*$ to denote the largest invariant ideal in $\Ext_D^{**}(\FF_2, \FF_2)$ contained in $I$. See Appendix A for more details.

The following lemma verifies that if $\p$ is an invariant prime ideal in $\Ext_D^{**}(\FF_2, \FF_2)$, then $\p^*$ is also prime. In other words, we have a well-defined map
    \begin{align*}
        \Spec^h(\Ext_D^{**}(\FF_2, \FF_2)) &\to \Specinv(\Ext_D^{**}(\FF_2, \FF_2)) \\
        \p &\mapsto \p^*.
    \end{align*}

\begin{lem}
    If $\p \in \Spec^h(\Ext_D^{**}(\FF_2, \FF_2))$, then $\p^* \in \Specinv(\Ext_D^{**}(\FF_2, \FF_2))$.
\end{lem}

\begin{proof}
    Since $(\pi^A_{**}(HD), \pi^A_{**}(HD \otimes HD)) \cong (\pi^A_{**}(HD), \pi^A_{**}(HD) \otimes (A \cotensor_D \FF_2))$ is a split Hopf algebroid with $A \cotensor_D \FF_2 \cong \FF_2[\xi_1^2, \xi_2^4, \xi_3^8, \ldots, \xi_n^{2^n}, \ldots]$ as algebras, the assumptions of Section 4 of \cite{landweber-associated} are met and Proposition 4.2(a) of \cite{landweber-associated} applies.
\end{proof}

\begin{lem}
    The map $\Spec^h(\Ext_D^{**}(\FF_2, \FF_2)) \to \Specinv(\Ext_D^{**}(\FF_2, \FF_2))$ given by $\p \mapsto \p^*$ is a retraction.
\end{lem}

\begin{proof}
    Since $\p^* = \p$ when $\p \in \Specinv(\Ext_D^{**}(\FF_2, \FF_2))$, it suffices to show that this map is continuous. Note that if $I$ is an invariant ideal in $\Ext_D^{**}(\FF_2, \FF_2)$ and $\p \in \Specinv(\Ext_D^{**}(\FF_2, \FF_2))$, then $I \subseteq \p$ if and only if $I \subseteq \p^*$, hence the preimage of $V^\textup{inv}(I)$ is $V(I)$.
\end{proof}

\begin{rem}
    The previous lemma allows us to view $\Specinv(\Ext_D^{**}(\FF_2, \FF_2))$ as a quotient of $\Spec^h(\Ext_D^{**}(\FF_2, \FF_2))$, and it seems that this may be more natural than to view it as a subspace. Speaking very informally (and ignoring grading issues), $\Specinv(\Ext_D^{**}(\FF_2, \FF_2))$ should be the underlying topological space of the stack associated to the Hopf algebroid $(\Ext_D^{**}(\FF_2, \FF_2), \Ext_D^{**}(\FF_2, \FF_2) \otimes (A \cotensor_D \FF_2))$. Since this is a split Hopf algebroid, the associated stack is the quotient stack $[\Spec^h(\Ext_D^{**}(\FF_2, \FF_2)) / \Spec^h(A \cotensor_D \FF_2)]$ coming from the coaction of the Hopf algebra $A \cotensor_D \FF_2$ on $\Ext_D^{**}(\FF_2, \FF_2)$.
\end{rem}

Next we check that the map $\Spc(\res_{A,D}) : \Spc(\Stable(D)^c) \to \Spc(\Stable(A)^c)$ factors through $\Specinv(\Ext_D^{**}(\FF_2, \FF_2))$.

\begin{lem}\label{lem:quotient}
    The map $\res_{A,D} : \Spc(\Stable(D)^c) \to \Spc(\Stable(A)^c)$ descends to the quotient to give a map $\wtilde{\varphi}_A : \Specinv(\Ext_D^{**}(\FF_2, \FF_2)) \to \Spc(\Stable(A)^c)$:
        \begin{center}
            \begin{tikzcd}[column sep = huge]
                \Spc(\Stable(D)^c) \ar[rr, "\Spc(\res_{A,D})"] \ar[d, "\cong"'] & & \Spc(\Stable(A)^c). \\
                \Spec^h(\Ext_D^{**}(\FF_2, \FF_2)) \ar[d] & & \\
                \Specinv(\Ext_D^{**}(\FF_2, \FF_2)) \ar[rruu, "\wtilde{\varphi}_A"', dashed] & &
            \end{tikzcd}
        \end{center}
\end{lem}

\begin{proof}    
    First note that if $X \in \Stable(A)^c$ and $\q \in \Spec^h(\Ext_E^{**}(\FF_2, \FF_2))$, then since $HE_{**}(X) \cong \pi_{**}^E(\res_{A,E}(X))$, we have
        \begin{align*}
            \q \in \sigma_E(\res_{A,E}(X)) &\iff \q \supseteq \sqrt{\ann_{\Ext_E^{**}(\FF_2, \FF_2)}(HE_{**}(X))} \\
            &\iff \q^* \supseteq \sqrt{\ann_{\Ext_E^{**}(\FF_2, \FF_2)}(HE_{**}(X))} \\
            &\iff \q^* \in \sigma_E(\res_{A,E}(X)),
        \end{align*}
    where we have used the fact that $\sqrt{\ann_{\Ext_E^{**}(\FF_2, \FF_2)}(HE_{**}(X))}$ is an invariant ideal by Proposition 6.3.2(b) of \cite{pal-stable}. Since the quotient map $D \twoheadrightarrow E$ induces a map of Hopf algebroids, it is straightforward to check that
        \[\res_{D,E}^*(\q^*) = (\res_{D,E}^*(\q))^*.\]

    Now suppose $\mathcal{P} \in \Spc(\Stable(D)^c)$, and let $\p \in \Spec^h(\Ext_D^{**}(\FF_2, \FF_2))$ denote the corresponding homogeneous prime ideal. Consider $\p^* \in \Specinv(\Ext_D^{**}(\FF_2, \FF_2)) \subseteq \Spec^h(\Ext_D^{**}(\FF_2, \FF_2))$ and the corresponding prime tt-ideal $\mathcal{P}^* \in \Spc(\Stable(D)^c)$. Since $\p^* \subseteq \p$, if $\p^* \in \sigma_D(\res_{A,D}(X))$ then $\p \in \sigma_D(\res_{A,D}(X))$. Conversely, if $\p \in \sigma_D(\res_{A,D}(X)) = \bigcup_{E \in \mathcal{Q}} \res_{D,E}^*(\sigma_E(\res_{A,E}(X)))$, then $\p = \res_{D,E}^*(\q)$ for some $\q \in \sigma_E(\res_{A,E}(X))$ and hence $\p^* = (\res_{D,E}^*(\q))^* = \res_{D,E}(\q^*)$ with $\q^* \in \sigma_E(\res_{A,E}(X))$, i.e., $\p^* \in \bigcup_{E \in \mathcal{Q}} \res_{D,E}^*(\sigma_E(\res_{A,E}(X))) = \sigma_D(\res_{A,D}(X))$. Therefore,
        \begin{align*}
            X \notin \Spc(\res_{A,D})(\mathcal{P}) &\iff \res_{A,D}(X) \notin \mathcal{P} \\
            &\iff \mathcal{P} \in \supp_D(\res_{A,D}(X)) \\
            &\iff \p \in \sigma_D(\res_{A,D}(X)) \\
            &\iff \p^* \in \sigma_D(\res_{A,D}(X)) \\
            &\iff \mathcal{P}^* \in \supp_D(\res_{A,D}(X)) \\
            &\iff \res_{A,D}(X) \notin \mathcal{P}^* \\
            &\iff X \notin \Spc(\res_{A,D})(\mathcal{P}^*).
        \end{align*}
    Thus, $\Spc(\res_{A,D})(\mathcal{P}) = \Spc(\res_{A,D})(\mathcal{P}^*)$, as desired.
\end{proof}

Now we work to show that $\wtilde{\varphi}_A$ is a homeomorphism. The next result is a corollary of Palmieri's nilpotence theorem.

\begin{prop}\label{prop:A-surj}
    The map
        \[\wtilde{\varphi}_A : \Specinv(\Ext_D^{**}(\FF_2, \FF_2)) \to \Spc(\Stable(A)^c)\]
    is surjective.
\end{prop}

\begin{proof}
    The stable morphism $\res_{A,D} : \Stable(A) \to \Stable(D)$ detects nilpotence by Theorem 5.1.5 of \cite{pal-stable}, therefore $\res_{A,D}^* : \Spc(\Stable(D)^c) \to \Spc(\Stable(A)^c)$ is surjective by Theorem 1.3 of \cite{bchs-surjectivity}. Thus, $\wtilde{\varphi}_A$ is surjective by Lemma \ref{lem:quotient}.
\end{proof}

We will now define a notion of support for $\Stable(A)^c$.

\begin{defn}
    Given $X \in \Stable(A)^c$, define
        \begin{align*}
            \wtilde{\sigma}_A(X) &= \sigma_D(\res_{A,D}(X)) \cap \Specinv(\Ext_D^{**}(\FF_2, \FF_2)) \\
            &= \left\{\p \in \Specinv(\Ext_D^{**}(\FF_2, \FF_2)) : HD_{**}(X)_\p \neq 0\right\}.
        \end{align*}
\end{defn}

\begin{rem}\label{rem:A-supp}
    Note that for $X \in \Stable(A)^c$, by the diagram defining $\wtilde{\sigma}_A$ in Lemma \ref{lem:quotient} we have
        \begin{align*}
            \wtilde{\varphi}_A^{-1}(\supp_A(X)) &= \rho_D(\supp_D(\res_{A,D}(X)) \cap \Specinv(\Ext_D^{**}(\FF_2, \FF_2)) \\
            &= \sigma_D(\res_{A,D}(X)) \cap \Specinv(\Ext_D^{**}(\FF_2, \FF_2)) \\
            &= \wtilde{\sigma}_A(X).
        \end{align*}
    where the second equality follows from Remark \ref{rem:D-supp}.
\end{rem}

Next we show that every quasi-compact open subset of $\Specinv(\Ext_D^{**}(\FF_2, \FF_2))$ (see Lemma \ref{lem:qc-opens}) arises as the $\wtilde{\sigma}_A$-support of a compact object, which is a corollary of Palmieri's periodicity theorem.

\begin{lem}\label{lem:A-iterated-cofiber}
    For any finitely generated invariant ideal $I$ in $\Ext_D^{**}(\FF_2, \FF_2)$, there exists an object $X \in \Stable(A)^c$ with $\wtilde{\sigma}_A(X) = V^\textup{inv}(I)$.
\end{lem}

\begin{proof}
    Since the coaction on $\Ext_D^{**}(\FF_2, \FF_2)$ lowers homological degree, we may rearrange the generators of $I$ to be in ascending order of homological degree and write $I = (f_1, \ldots, f_k)$ such that $f_j$ is invariant mod $(f_1, \ldots, f_{j-1})$ for each $j = 1, \ldots, k$.
    
    We will inductively define $X_j$ with
        \[\wtilde{\sigma}_A(X_j) = V^\textup{inv}\left(\sqrt{\ann_{\Ext_D^{**}(\FF_2, \FF_2)}(HD_{**}(X_j))}\right) = V^\textup{inv}((f_1, \ldots, f_j))\]
    for $j = 0, \ldots, k$ as follows. First, let $X_0 = S^0_A$, which has $\wtilde{\sigma}_A(X_0) = V^\textup{inv}((0))$. Now suppose $X_{j-1}$ has been defined. By the inductive hypothesis, $\sqrt{(f_1, \ldots, f_{j-1})} = \sqrt{\ann_{\Ext_D^{**}(\FF_2, \FF_2)}(HD_{**}(X_{j-1}))}$.
    Since $f_j$ is invariant mod $(f_1, \ldots, f_{j-1})$, it is also invariant mod $\sqrt{\ann_{\Ext_D^{**}(\FF_2, \FF_2)}(HD_{**}(X_{j-1}))}$.
    Therefore, by Theorem 6.1.3 of \cite{pal-stable}, there is a self-map of $X_{j-1}$ such that the induced map on $HD$-homology is multiplication by $f_j^{n_j}$ for some $n_j$. Take $X_j$ to be the cofiber of this map. We see from the long exact sequence
        \begin{center}
            \begin{tikzcd}
                \cdots \ar[r] & HD_{**}(X_{j-1}) \ar[r, "f_j^{n_j}"] & HD_{**}(X_{j-1}) \ar[r] & HD_{**}(X_j) \ar[r] & \cdots
            \end{tikzcd}
        \end{center}
    that $\p \in \wtilde{\sigma}_A(X_j)$ if and only if $\p \in \wtilde{\sigma}_A(X_{j-1})$ and $f_j^{n_j}$ is not a unit in $\Ext_D^{**}(\FF_2, \FF_2)_\p$, i.e., $f_j \in \p$. Hence,
        \begin{align*}
            \wtilde{\sigma}_A(X_j) &= \wtilde{\sigma}_A(X_{j-1}) \cap V^\textup{inv}((f_j)) \\
            &= V^\textup{inv}((f_1, \ldots, f_{j-1})) \cap V^\textup{inv}((f_j)) \\
            &= V^\textup{inv}((f_1, \ldots, f_j)).
        \end{align*}
    Moreover, we see by the five lemma that
        \[(f_1, \ldots, f_{j-1}) \subseteq \sqrt{\ann_{\Ext_D^{**}(\FF_2, \FF_2)}(HD_{**}(X_j))},\]
    and chasing through the diagram
        \begin{center}
            \begin{tikzcd}
                \cdots \ar[r] & HD_{**}(X_{j-1}) \ar[r] \ar[d, "f_j^{n_j}"'] & HD_{**}(X_j) \ar[r] \ar[d, "f_j^{n_j}"] \ar[rd, bend left =  13, "0"'] & HD_{**}(X_{j-1}) \ar[r] \ar[d, "f_j^{n_j}"] & \cdots \\
                \cdots \ar[r] & HD_{**}(X_{j-1}) \ar[r] \ar[d, "f_j^{n_j}"'] \ar[rd, bend right = 13, "0"] & HD_{**}(X_j) \ar[r] \ar[d, "f_j^{n_j}"] & HD_{**}(X_{j-1}) \ar[r] \ar[d, "f_j^{n_j}"] & \cdots \\
                \cdots \ar[r] & HD_{**}(X_{j-1}) \ar[r] & HD_{**}(X_j) \ar[r] & HD_{**}(X_{j-1}) \ar[r] & \cdots
            \end{tikzcd}
        \end{center}
    shows that $f_j \in \sqrt{\ann_{\Ext_D^{**}(\FF_2, \FF_2)}(HD_{**}(X_j))}$. Therefore,
        \[V^\textup{inv}\left(\sqrt{\ann_{\Ext_D^{**}(\FF_2, \FF_2)}(HD_{**}(X_j))}\right) \subseteq V^\textup{inv}((f_1, \ldots, f_j)).\]
    Also, if $\p \not\supseteq \sqrt{\ann_{\Ext_D^{**}(\FF_2, \FF_2)}(HD_{**}(X_j))}$, then $X_\p = 0$, so
        \[\wtilde{\sigma}_A(X_j) \subseteq V^\textup{inv}\left(\sqrt{\ann_{\Ext_D^{**}(\FF_2, \FF_2)}(HD_{**}(X_j))}\right).\]
    Altogether, we have
        \[\wtilde{\sigma}_A(X_j) = V^\textup{inv}\left(\sqrt{\ann_{\Ext_D^{**}(\FF_2, \FF_2)}(HD_{**}(X_j))}\right) = V^\textup{inv}((f_1, \ldots, f_j)),\]
    completing the induction step and finishing the proof by taking $X = X_k$.
\end{proof}

We now arrive at our main result.

\begin{thm}\label{thm:Spc(Stable(A)^c)}
    The map
        \[\wtilde{\varphi}_A : \Specinv(\Ext_D^{**}(\FF_2, \FF_2)) \to \Spc(\Stable(A)^c)\]
    is a homeomorphism.
\end{thm}

\begin{proof}
    If $Z \subseteq \Specinv(\Ext_D^{**}(\FF_2, \FF_2))$ is a closed subset with quasi-compact complement, then we have $Z = V^\textup{inv}(I)$ for some finitely generated invariant ideal $I$ in $\Ext_D^{**}(\FF_2, \FF_2)$ by Lemma \ref{lem:qc-opens}. By Remark \ref{rem:A-supp} and Lemma \ref{lem:A-iterated-cofiber}, there exists an object $X \in \Stable(A)^c$ with
        \[\wtilde{\varphi}_A^{-1}(\supp_A(X)) = \wtilde{\sigma}_A(X) = V^\textup{inv}(I) = Z.\]
    Moreover, $\wtilde{\varphi}_A$ is surjective by Proposition \ref{prop:A-surj}. Thus, by Proposition \ref{prop:tt-result}, $\wtilde{\varphi}_A$ is a homeomorphism.
\end{proof}

We may also write the Balmer spectrum in terms of the elementary quotient Hopf algebras of $A$. Let $\mathcal{Q}_0$ denote the full subcategory of $\mathcal{Q}$ consisting of the conormal elementary quotient Hopf algebras of $A$, which is final in $\mathcal{Q}$. For each $E \in \mathcal{Q}_0$, let us write $\Specinv(\Ext_E^{**}(\FF_2, \FF_2))$ to denote the space of invariant prime ideals of the Hopf algebroid
    \[\left(\pi^A_{**}(HE), \pi^A_{**}(HE \otimes HE)\right) \cong \left(\Ext_E^{**}(\FF_2, \FF_2), \Ext_E^{**}(\FF_2, \FF_2) \otimes (A \cotensor_E \FF_2)\right).\]
Note that each quotient map $E \twoheadrightarrow E'$ in $\mathcal{Q}_0$ induces a map of Hopf algebroids, which in turn induces a map on $\Specinv$.

\begin{cor}\label{cor:A-quillen}
        There is a homeomorphism
        \[\Spc(\Stable(A)^c) \cong \colim_{E \in \mathcal{Q}_0} \Specinv(\Ext_E^{**}(\FF_2, \FF_2)).\]
\end{cor}

\begin{proof}
    For each $E \in \mathcal{Q}_0$, the quotient map $D \twoheadrightarrow E$ induces a map of Hopf algebroids, which in turn induces a map on $\Specinv$. Therefore, we have a map
        \[\colim_{E \in \mathcal{Q}_0} \Specinv(\Ext_E^{**}(\FF_2, \FF_2)) \to \Specinv(\Ext_D^{**}(\FF_2, \FF_2)).\]
    It suffices to show that this map is a homeomorphism since $\Spc(\Stable(A)^c) \cong \Specinv(\Ext_D^{**}(\FF_2, \FF_2))$ by Theorem \ref{thm:Spc(Stable(A)^c)}. Consider the following commutative diagram
        \begin{center}
            \begin{tikzcd}
                \displaystyle\colim_{E \in \mathcal{Q}_0} \Specinv(\Ext_E^{**}(\FF_2, \FF_2)) \ar[r] \ar[d, hook] & \Specinv(\Ext_D^{**}(\FF_2, \FF_2)) \ar[d, hook] \\
                \displaystyle\colim_{E \in \mathcal{Q}_0} \Spec^h(\Ext_E^{**}(\FF_2, \FF_2)) \ar[r, "\cong"] \ar[d, two heads] & \Spec^h(\Ext_D^{**}(\FF_2, \FF_2)) \ar[d, two heads] \\
                \displaystyle\colim_{E \in \mathcal{Q}_0} \Specinv(\Ext_E^{**}(\FF_2, \FF_2)) \ar[r] & \Specinv(\Ext_D^{**}(\FF_2, \FF_2)),
            \end{tikzcd}
        \end{center}
    where the middle horizontal map is a homeomorphism by Proposition 4.2 of \cite{pal-quillen}. By the top square, the desired map is a homeomorphism onto its image, and by the bottom square it is surjective.
\end{proof}

\begin{rem}
    Note that $\Specinv(\Ext_D^{**}(\FF_2, \FF_2))$ is a quotient of $\Spc(\Stable(D)^c) \cong \Spec^h(\Ext_D^{**}(\FF_2, \FF_2))$. It would be nice to describe this tt-geometrically in terms of some sort of ``coaction'' of $A \cotensor_D \FF_2$ on the tt-category $\Stable(D)^c$. This seems to be related to the idea of descent and Hopf-Galois extensions (e.g., see \cite{bchnps-descent} and \cite{rognes-galois}).
\end{rem}

Lastly, we show that the comparison map for $\Stable(A)^c$ is \emph{not} a homeomorphism. More precisely, we have the following result.

\begin{prop}
    The comparison map
        \[\rho_A : \Spc(\Stable(A)^c) \to \Spec^h(\Ext_A^{**}(\FF_2, \FF_2))\]
    is surjective but not injective.
\end{prop}

\begin{proof}
    To see that $\rho_A$ is surjective, note that the following diagram commutes by naturality of comparison maps
        \begin{center}
            \begin{tikzcd}
                \Spc(\Stable(D)^c) \ar[r, "\res_{A,D}^*"] \ar[d, "\rho_D"'] & \Spc(\Stable(A)^c) \ar[d, "\rho_A"] \\
                \Spec^h(\Ext_D^{**}(\FF_2, \FF_2)) \ar[r, "\res_{A,D}^*"] & \Spec^h(\Ext_A^{**}(\FF_2, \FF_2)).
            \end{tikzcd}
        \end{center}
    The bottom horizontal map is surjective by Lemma 5.6 of \cite{pal-quillen} and $\rho_D$ is a homeomorphism by Theorem \ref{thm:Spc(Stable(D)^c)}, therefore $\rho_A$ is surjective.
    
    To see that $\rho_A$ is not injective, notice that after examining the following commutative diagram
        \begin{center}
            \begin{tikzcd}
                \Spc(\Stable(D)^c) \ar[rr, "\Spc(\res_{A,D})"] \ar[dd, "\rho_D"'] & & \Spc(\Stable(A)^c) \ar[dd, "\rho_A"] \\
                & \Specinv(\Ext_D^{**}(\FF_2, \FF_2)) \ar[ru, "\cong"] & \\
                \Spec^h(\Ext_D^{**}(\FF_2, \FF_2)) \ar[ru] \ar[rr, "\res_{A,D}^*"] & & \Spec^h(\Ext_A^{**}(\FF_2, \FF_2)),
            \end{tikzcd}
        \end{center}
    we may identify $\rho_A$ with the composition
        \[\Specinv(\Ext_D^{**}(\FF_2, \FF_2)) \hookrightarrow \Spec^h(\Ext_D^{**}(\FF_2, \FF_2)) \xrightarrow{\res_{A,D}^*} \Spec^h(\Ext_A^{**}(\FF_2, \FF_2)).\]
    Now recall the elementary quotient Hopf algebra $E(0) = A/(\xi_1^2, \xi_2^2, \xi_3^2, \ldots)$. The quotient map $D \twoheadrightarrow E$ induces a map of Hopf algebroids, which in turn induces a map on $\Specinv$, resulting in the following commutative diagram
        \begin{center}
            \begin{tikzcd}[column sep = 0.26in]
                \Specinv(\Ext_{E(0)}^{**}(\FF_2, \FF_2)) \ar[r] \ar[d, hook] & \Specinv(\Ext_D^{**}(\FF_2, \FF_2)) \ar[d, hook] \ar[rd] \\
                \Spec^h(\Ext_{E(0)}^{**}(\FF_2, \FF_2)) \ar[r, "\res_{D,E(0)}^*"'] & \Spec^h(\Ext_D^{**}(\FF_2, \FF_2)) \ar[r, "\res_{A,D}^*"'] & \Spec^h(\Ext_A^{**}(\FF_2, \FF_2)).
            \end{tikzcd}
        \end{center}
    Consider $\p_1 = \res_{D,E(0)}^*((h_{10})), \p_2 = \res_{D,E(0)}^*((h_{10}, h_{20})) \in \Specinv(\Ext_D^{**}(\FF_2, \FF_2))$. Recall from the proof of Proposition \ref{prop:res-of-loc} that the map $\res_{D,E(0)}$ is ``surjective up to nilpotents'' (i.e., for every $y$ in the codomain, $y^{2^n}$ is in the image for some $n$), so $\res_{D,E(0)}^*$ is injective and hence $\p_1 \neq \p_2$. On the other hand, the map $\res_{A,E(0)} : \Ext_A^{**}(\FF_2, \FF_2) \to \Ext_{E(0)}^{**}(\FF_2, \FF_2)$ factors through the invariants
        \[\Ext_{E(0)}^{**}(\FF_2, \FF_2)^{A \cotensor_{E(0)} \FF_2} \cong \FF_2[h_{10}]\]
    (a fact which was probably known to Adams sometime in the 1960s), so the preimage of any ideal in $\Ext_{E(0)}^{**}(\FF_2, \FF_2)$ containing $h_{10}$ is the maximal ideal in $\Ext_A^{**}(\FF_2, \FF_2)$. In particular,
        \[\res_{A,D}^*(\p_1) = \res_{A,E(0)}^*((h_{10})) = \res_{A,E(0)}^*((h_{10}, h_{20})) = \res_{A,D}^*(\p_2).\]
    Thus, $\rho_A$ is not injective.
\end{proof}


\appendix

\section{Invariant prime ideals of a Hopf algebroid}

Let $(A, \Gamma)$ be a bigraded commutative Hopf algebroid over a field $k$. We will make the following additional assumptions:
    \begin{itemize}
        \item $k$ is of characteristic 2 (this assumption can be removed if $A$ is concentrated in even degrees).
        \item $(A, \Gamma)$ is connected, i.e., $\Gamma_{ij} = 0$ for $i < 0$ or $j < 0$, and $\Gamma_{00} \cong A$.
        \item $\Gamma$ is of finite type, i.e., $\Gamma_{ij}$ is finite-dimensional over $k$ for each $i,j$.
    \end{itemize}
The standard reference for Hopf algebroids is Appendix A1 of \cite{ravenel-complex}.

\begin{rem}
    We will primarily be interested in split Hopf algebroids, which arise in the following manner. Given a $k$-algebra $A$ and a commutative Hopf algebra $H$ over $k$ such that $A$ has a right $H$-coaction, there is a Hopf algebroid $(A, A \otimes H)$ whose left unit $\eta_L$ is given by $a \mapsto a \otimes 1$ and whose right unit $\eta_R : A \to A \otimes H$ is given by the $H$-coaction on $A$.
\end{rem}

\begin{defn}\label{defn:invariant-ideal}
    A homogeneous ideal $I$ in $A$ is \emph{invariant} if $\eta_R(I) \subseteq I \Gamma$, where $\eta_R : A \to \Gamma$ is the right unit map of $(A, \Gamma)$.
\end{defn}

Note that if $(A, \Gamma) \cong (A, A \otimes H)$ is a split Hopf algebroid coming from a right coaction $\psi : A \to A \otimes H$, then an ideal $I$ in $A$ is invariant if and only if $\psi(I) \subseteq I \otimes H$.

\begin{defn}\label{defn:Specinv}
    Let $\Specinv(A)$ denote the subspace of the bihomogeneous Zariski spectrum $\Spec^h(A)$ consisting of the invariant prime ideals in $A$. Given a homogeneous ideal $I$ in $A$, let
        \[V^\textup{inv}(I) = V(I) \cap \Specinv(A).\]
\end{defn}

In \cite{landweber-associated}, Landweber looks at operations which convert homogeneous ideals into invariant ideals for certain types of split Hopf algebroids. These operations extend easily to non-split Hopf algebroids.

\begin{defn}
    Given a homogeneous ideal $I$ in $A$, let $I^*$ denote the largest invariant ideal in $A$ contained in $I$ and let $I^\sharp$ denote the smallest invariant ideal in $A$ containing $I$.
\end{defn}

\begin{prop}\label{prop:invariant-existence}
    Suppose $I$ is a homogeneous ideal in $A$.
        \begin{enumerate}[(i)]
            \item $I^*$ exists and is given explicitly by
                \[I^* = \{x \in I : \eta_R(x) \in I \Gamma\}.\]
            \item $I^\sharp$ exists and can be described as follows. Fix a homogeneous basis $\{\gamma_i\}$ for $\Gamma$, and for each $x \in I$ write $\eta_R(x) = \sum_i x_i \gamma_i$. Then $I^\sharp$ is generated by all of the $x_i$ as $x$ ranges over (a generating set for) $I$.
        \end{enumerate}
\end{prop}

\begin{proof}
    For part (i), note that $\{x \in I : \eta_R(x) \in I \Gamma\}$ is an ideal since $\eta_R$ is $A$-linear, and it is invariant since $(\eta_R \otimes \id_\Gamma) \circ \eta_R = (\id_A \otimes \Delta ) \circ \eta_R$ as maps $A \to A \otimes_A \Gamma \otimes_A \Gamma$. If $I'$ is an any other invariant ideal contained in $I$, then $\eta_R(x) \in I' \Gamma \subseteq I \Gamma$ for every $x \in I'$. Thus, $\{x \in I : \eta_R(x) \in I \Gamma\}$ is the largest invariant ideal contained in $I$.

    For part (ii),
    the given description defines a homogeneous ideal which is invariant for similar reasons as above, and it is clearly the smallest possible invariant ideal that contains $I$.
\end{proof}

\begin{lem}\label{lem:invariant-properties}
        \begin{enumerate}[(i)]
            \item If $I$ is a homogeneous ideal in $A$, then $I$ is invariant if and only if $I = I^*$ if and only if $I = I^\sharp$.
            \item If $I \subseteq J$ are homogeneous ideals in $A$, then $I^* \subseteq J^*$ and $I^\sharp \subseteq J^\sharp$.
            \item If $\{I_\alpha\}$ is a collection of invariant ideals in $A$, then $\bigcap_\alpha I_\alpha$ is invariant.
            \item If $I$ and $J$ are invariant ideals in $A$, then $I + J$ is invariant.
            \item If $I$ and $J$ are invariant ideals in $A$, then $IJ$ is invariant.
            \item If $I$ is an invariant ideal in $A$, then $\sqrt{I}$ is the intersection of all invariant prime ideals containing $I$, which is therefore invariant.
            \item If $I$ is a finitely generated homogeneous ideal in $A$, then $I^\sharp$ is finitely generated.
        \end{enumerate}
\end{lem}

\begin{proof}    
    Parts (i)-(iii) are evident. Parts (iv) and (v) follow from the fact that $\eta_R$ is a map of algebras.
    
    For part (vi), first note that $\sqrt{I} = \bigcap_{\p \in V(I)} \p \subseteq \bigcap_{\p \in V^\textup{inv}(I)} \p$.
    Since $I$ is invariant, $I \subseteq \p$ if and only if $I \subseteq \p^*$ for any $\p \in \Spec^h(A)$, hence
        \[\bigcap_{\p \in V^\textup{inv}(I)} \p = \bigcap_{\p \in V(I)} \p^* \subseteq \bigcap_{\p \in V(I)} \p = \sqrt{I}.\]

    Part (vii) follows from the description of $I^\sharp$ given in part (ii) of Proposition \ref{prop:invariant-existence} and the assumptions that $(A, \Gamma)$ is connected and $\Gamma$ is of finite type.
\end{proof}

\begin{rem}
    Note that if $I$ is a homogeneous ideal in $A$ and $\p \in \Specinv(A)$, then $I \subseteq \p$ if and only if $I^\sharp \subseteq \p$ and therefore
        \[V^\textup{inv}(I^\sharp) = V^\textup{inv}(I).\]
    As a result, every closed subset of $\Specinv(A)$ can be written as $V^\textup{inv}(I)$ for some invariant ideal $I$ in $A$.
\end{rem}

Recall the following definition.

\begin{defn}
    A topological space is \emph{spectral} if it is quasi-compact and $T_0$, it has a basis of quasi-compact open subsets which is closed under finite intersections, and every irreducible closed subset has a unique generic point.
\end{defn}

We will show that $\Specinv(A)$ is spectral through a series of lemmas. We start by characterizing the quasi-compact open subsets of $\Specinv(A)$.

\begin{lem}\label{lem:qc-opens}
    An open subset $U \subseteq \Specinv(A)$ is quasi-compact if and only if
        \[U = \Specinv(A) \ssm V^\textup{inv}(I)\]
    for some finitely generated invariant ideal $I$ in $A$.
\end{lem}

\begin{proof}
    First suppose $I = (x_1, \ldots, x_n)$ is a finitely generated invariant ideal in $A$ and consider $U = \Specinv(A) \ssm V^\textup{inv}(I)$. Suppose we have an open cover $U \subseteq \bigcup_\alpha U_\alpha$, where $U_\alpha = \Specinv(A) \ssm V^\textup{inv}(I_\alpha)$. Then
        \[V^\textup{inv}(I) \supseteq \bigcap_\alpha V^\textup{inv}(I_\alpha) = V^\textup{inv}\left(\sum_\alpha I_\alpha\right),\]
    so $I \subseteq \sqrt{\sum_{\alpha \in \mathcal{A}} I_\alpha}$. Therefore, there exist $\alpha_1, \ldots, \alpha_k$ and $m_1, \ldots, m_n$ such that $x_1^{m_1}, \ldots, x_n^{m_n} \in I_{\alpha_1} + \cdots + I_{\alpha_k}$. But then
        \[I \subseteq \sqrt{(x_1^{m_1}, \ldots, x_n^{m_n})} \subseteq \sqrt{I_{\alpha_1} + \cdots + I_{\alpha_k}},\]
    so
        \[V^\textup{inv}(I) \supseteq V^\textup{inv}(I_{\alpha_1} + \cdots + I_{\alpha_k}) = V^\textup{inv}(I_{\alpha_1}) \cap \cdots \cap V^\textup{inv}(I_{\alpha_k}),\]
    i.e., $U \subseteq U_{\alpha_1} \cup \cdots \cup U_{\alpha_k}$. Thus, $U$ is quasi-compact.

    Conversely, suppose $U$ is a quasi-compact open subset $\Specinv(A)$. Then $U = \Specinv(A) \ssm V^\textup{inv}(I)$ for some invariant ideal $I$ in $A$. Note that $I$ is the sum of its finitely generated subideals, so by applying $(-)^\sharp$, we see by Lemma \ref{lem:invariant-properties}(vii) that $I$ is the sum of its finitely generated invariant subideals. Let $\mathcal{F}$ denote the set of finitely generated invariant ideals contained in $I$, so that
        \begin{align*}
            U &= \Specinv(A) \ssm V^\textup{inv}(I) \\
            &= \Specinv(A) \ssm V^\textup{inv}\left(\sum_{J \in \mathcal{F}} J\right) \\
            &= \Specinv(A) \ssm \bigcap_{J \in \mathcal{F}} V^\textup{inv}(J) \\
            &= \bigcup_{J \in \mathcal{F}} \Specinv(A) \ssm V^\textup{inv}(J).
        \end{align*}
    Since $U$ is quasi-compact, there exist $J_1, \ldots, J_n \in \mathcal{F}$ such that
        \begin{align*}
            U &= \left(\Specinv(A) \ssm V^\textup{inv}(J_1)\right) \cup \cdots \cup \left(\Specinv(A) \ssm V^\textup{inv}(J_n)\right) \\
            &= \Specinv(A) \ssm \left(V^\textup{inv}(J_1) \cap \cdots \cap V^\textup{inv}(J_n)\right) \\
            &= \Specinv(A) \ssm V^\textup{inv}(J_1 + \cdots + J_n),
        \end{align*}
    where $J_1 + \cdots + J_n$ is a finitely generated invariant ideal in $A$ by Lemma \ref{lem:invariant-properties}(iv).
\end{proof}

\begin{lem}\label{lem:qc-t0}
    $\Specinv(A)$ is quasi-compact and $T_0$.
\end{lem}

\begin{proof}
    $\Specinv(A)$ is compact by Lemma \ref{lem:qc-opens}. To show that $\Specinv(A)$ is $T_0$, consider $\p_1 \neq \p_2$ in $\Specinv(A)$. Without loss of generality, assume $\p_1 \not\subseteq \p_2$. Then taking $U = \Specinv(A) \ssm V^\textup{inv}(\p_1)$, we have $\p_1 \notin U$ and $\p_2 \in U$.
\end{proof}

\begin{lem}\label{lem:qc-basis}
    The quasi-compact open subsets give a basis for the topology on $\Specinv(A)$.
\end{lem}

\begin{proof}
    Suppose $\p \in \Specinv(A) \ssm V(I)$ for some invariant ideal $I$ in $A$. Since $I \not\subseteq \p$, choose $x \in I \ssm \p$. Then $(x)^\sharp$ is a finitely generated invariant ideal with $(x)^\sharp \subseteq I$ and $(x)^\sharp \not\subseteq \p$, i.e.,
        \[\p \in \Specinv(A) \ssm V^\textup{inv}\left((x)^\sharp\right) \subseteq \Specinv(A) \ssm V^\textup{inv}(I).\]
\end{proof}

\begin{lem}\label{lem:qc-intersection}
    The quasi-compact open subsets of $\Specinv(A)$ are closed under finite intersections.
\end{lem}

\begin{proof}
    Consider $U = \Specinv(A) \ssm V^\textup{inv}(I)$ and $V = \Specinv(A) \ssm V^\textup{inv}(J)$ for finitely generated invariant ideals $I$ and $J$ in $A$. Then
        \[U \cap V = \Specinv(A) \ssm \left(V^\textup{inv}(I) \cup V^\textup{inv}(J)\right) = \Specinv(A) \ssm V^\textup{inv}(IJ),\]
    where $IJ$ is finitely generated and invariant by Lemma \ref{lem:invariant-properties}(v).
\end{proof}

\begin{lem}\label{lem:sober}
    $\Specinv(A)$ is sober, i.e., every nonempty irreducible closed subset has a generic point.
\end{lem}

\begin{proof}
    Suppose $V^\textup{inv}(I)$ is irreducible for some proper invariant ideal $I$ in $A$. By replacing $I$ with $\sqrt{I}$ if necessary (which is also invariant by Lemma \ref{lem:invariant-properties}(vi)), we may assume that $I$ is radical. For any $\p \in \Specinv(A)$, the closure of $\{\p\}$ is $V^\textup{inv}(\p)$, therefore it suffices to show that $I$ is prime (note that the generic point is necessarily unique). If $xy \in I$, then
        \[V^\textup{inv}(I) \subseteq V^\textup{inv}((xy)) = V^\textup{inv}((x)) \cup V^\textup{inv}((y)).\]
    Since $V^\textup{inv}(I)$ is irreducible, either $V^\textup{inv}\left((x)\right)$ or $V^\textup{inv}\left((y)\right)$ must contain $V^\textup{inv}(I)$, i.e., either $x \in I$ or $y \in I$ since $I$ is radical. Thus, $I$ is prime, as desired.
\end{proof}

The main result of this appendix now follows immediately from Lemmas \ref{lem:qc-t0}, \ref{lem:qc-basis}, \ref{lem:qc-intersection}, and \ref{lem:sober}.

\begin{prop}\label{prop:spectral}
    $\Specinv(A)$ is a spectral space.
\end{prop}


\printbibliography

\end{document}